\documentclass[10 pt, reqno]{amsart}
\usepackage[utf8]{inputenc}

\usepackage{new-style}
\newcommand{\bbone}[1]{\mathbbm{1}\{#1\}}
\newcommand{\set}[1]{\{#1\}}
\newcommand{\defeq}{\coloneqq}
\newcommand{\N}{\mathbb{N}}

\newcommand{\E}{\mathbb{E}}

\newcommand{\eq}{\mathsf{eq}}

\newcommand{\0}{\varnothing}

\newcommand{\Bad}{\mathsf{Bad}}
\newcommand{\Good}{\mathsf{Good}}

\usepackage[backend=biber, sorting=nyt, maxnames=100,backref=true, giveninits=true, sortcites=true]{biblatex}
\title{Independent sets and colorings of $K_{t,t,t}$-free graphs}

\author{Abhishek Dhawan$^{\star}$}
\address{$^{\star}$Department of Mathematics, University of Illinois Urbana--Champaign}
\email{adhawan2@illinois.edu}
\thanks{Abhishek Dhawan received funding from NSF RTG grant DMS-1937241.}

\author{Oliver Janzer$^{\dagger}$}
\address{$^{\dagger}$Institute of Mathematics, \'Ecole Polytechnique F\'ed\'erale de Lausanne}
\email{oliver.janzer@epfl.ch}

\author{Abhishek Methuku$^{\star}$}
\email{methuku@illinois.edu}
\thanks{Abhishek Methuku is supported by the UIUC Campus Research Board Award RB25050.}

\date{}

\begin{document}

\begin{abstract}
Alon, Krivelevich, and Sudakov conjectured in 1999 that every $F$-free graph of maximum degree at most $\Delta$ has chromatic number $O(\Delta / \log \Delta)$. This was previously known only for almost bipartite graphs, that is, for subgraphs of $K_{1,t,t}$ (verified by Alon, Krivelevich, and Sudakov themselves), while most recent results were concerned with improving the leading constant factor in the case where $F$ is almost bipartite. We prove this conjecture for all $3$-colorable graphs $F$, i.e. subgraphs of $K_{t,t,t}$, representing the first progress toward the conjecture since it was posed.

A closely related conjecture of Ajtai, Erd\H{o}s, Koml\'os, and Szemer\'edi from 1981 asserts that for every graph $F$, every $n$-vertex $F$-free graph of average degree $d$ contains an independent set of size $\Omega(n \log d / d)$. 
We prove this conjecture in a strong form for all 3-colorable graphs $F$. More precisely, we show that every $n$-vertex $K_{t,t,t}$-free graph of average degree $d$ contains an independent set of size at least $(1 - o(1)) n \log d / d$, matching Shearer's celebrated bound for triangle-free graphs (the case $t = 1$) and thereby yielding a substantial strengthening of it. Our proof combines a new variant of the R\"odl nibble method for constructing independent sets with a Tur\'an-type result on $K_{t,t,t}$-free graphs.

\end{abstract}

\maketitle

\sloppy

\section{Introduction}\label{sec: intro}

\subsection{Background}

The \emph{independence number} $\alpha(G)$ of a graph $G$ is the size of its largest independent set, while the \emph{chromatic number} $\chi(G)$ is the minimum number of colors required to color its vertices so that adjacent vertices receive distinct colors.

It is well known that if a graph $G$ has maximum degree $\Delta(G)$ and average degree $d(G)$, then $\alpha(G) \ge \frac{n}{d(G) + 1}$ and $\chi(G) \le \Delta(G) + 1$. These are commonly referred to as the \emph{greedy bounds}. A central area of research in graph theory is to understand under what structural constraints these bounds can be improved. In this paper, we focus on the case of $F$-free graphs—graphs that contain no subgraph (not necessarily induced) isomorphic to a fixed graph $F$.

Ajtai, Koml\'os, and Szemer\'edi first studied this problem for $F = K_3$. They showed \cite{ajtai1980note} that an $n$-vertex $K_3$-free graph $G$ with average degree $d$ satisfies $\alpha(G) = \Omega\!\left(\frac{n \log d}{d}\right)$. For $r \ge 3$, the same set of authors, together with Erd\H{o}s, proved \cite{ajtai1981turan} that $\alpha(G) = \Omega\!\left(\frac{n \log(\log d / r)}{r d}\right)$ for $K_r$-free graphs $G$. In light of these results, they posed the following conjecture:

\begin{conjecture}[{Ajtai--Erd\H{o}s--Koml\'os--Szemer\'edi \cite{ajtai1981turan}}]\label{conj: AEKS}
   For every fixed graph $F$, there exists a constant $c_1(F) > 0$ such that every $n$-vertex $F$-free graph $G$ with average degree at most $d$ satisfies\footnote{Throughout this paper, we use the asymptotic notation $o_x(\cdot)$ as $x \to \infty$.
We ignore the subscript when $x$ is clear from context.}
\[
\alpha(G) \ge (c_1(F) - o_d(1)) \frac{n \log d}{d}.
\]
\end{conjecture}

If true, the above conjecture is optimal up to the constant $c_1(\cdot)$, as there exist $n$-vertex, $d$-regular graphs of arbitrarily high girth whose independence number is at most $(2 + o_d(1)) \frac{n \log d}{d}$~\cite{BollobasIndependence}. A breakthrough of Shearer~\cite{shearer1995independence} established in 1995 that every $K_r$-free graph $G$ of average degree $d$ satisfies $\alpha(G) = \Omega\!\left(\frac{n \log d}{r d \log \log d}\right)$, which remains the best general bound known to date  (see also~\cite{bernshteyn2019johansson, Molloy, DKPS, dhawan2025bounds, bonamy2022bounding} and the survey~\cite[Section~3]{KangKelly2023nibble}). For the special case of triangle-free graphs $G$, Shearer proved the substantially stronger estimate $\alpha(G) \ge (1 - o_d(1)) \frac{n \log d}{d}$ in his celebrated 1983 paper~\cite{Shearer1983}, which in particular implies that $c_1(F) \ge 1$ when $F$ is a triangle. The largest class of graphs for which Conjecture~\ref{conj: AEKS} is known to hold consists of the so-called \emph{almost bipartite graphs}, that is, subgraphs of $K_{1,t,t}$ for $t \in \mathbb{N}$. 
In 1999, Alon, Krivelevich, and Sudakov deduced the bound $c_1(K_{1,t,t}) = \Omega(1/t)$ from a stronger result concerning the chromatic number.

\begin{theorem}[{Alon--Krivelevich--Sudakov \cite[Corollary~2.4]{AKSConjecture}}]\label{theorem: aks}
    Let $G$ be a $K_{1, t, t}$-free graph of maximum degree at most $\Delta$.
    Then, $\chi(G) = O\left(\frac{t\Delta}{\log \Delta}\right).$
\end{theorem}

They obtained the above result as a corollary of a more general statement concerning \emph{locally sparse} graphs. More precisely, a graph $G$ with maximum degree $\Delta$ is said to be \emph{$f$-locally sparse} if, for every vertex $v$, the subgraph $G[N(v)]$ spans at most $\Delta^2 / f$ edges. They proved that such graphs satisfy $\chi(G) = O(\Delta / \log f)$ \cite[Theorem~1.1]{AKSConjecture}. This immediately yields Theorem~\ref{theorem: aks} since every $K_{1,t,t}$-free graph with maximum degree $\Delta$ is $\Delta^{\Omega(1/t)}$-locally sparse by the K\H{o}v\'ari--S\'os--Tur\'an theorem \cite{KovariSosTuran}. 

Motivated by Theorem~\ref{theorem: aks}, Alon, Krivelevich, and Sudakov proposed the following strengthening of Conjecture~\ref{conj: AEKS}:

\begin{conjecture}[{Alon--Krivelevich--Sudakov \cite{AKSConjecture}}]\label{conj: AKS}
    For every fixed graph $F$, there exists a constant $c_2(F) > 0$ such that every $F$-free graph $G$ with maximum degree at most $\Delta$ satisfies 
\[
\chi(G) \le (c_2(F) + o_{\Delta}(1)) \frac{\Delta}{\log \Delta}.
\]
\end{conjecture}

Johansson extended Shearer's result on the independence number of $K_{r}$-free graphs to colorings, showing that such graphs satisfy $\chi(G) = O\left(\frac{r\Delta\log\log\Delta}{\log\Delta}\right)$ \cite{J96-Kr} (see \cite{Molloy, bernshteyn2019johansson, dhawan2025bounds} for more recent proofs with improvements in terms of the leading constant factor; see also \cite{bonamy2022bounding, DKPS} for an asymptotic improvement in the regime $r = \omega\left(\frac{\log \Delta}{\log \log \Delta}\right)$).
While almost-bipartite graphs remain the only class of graphs for which Conjectures~\ref{conj: AEKS}~and~\ref{conj: AKS} are known to hold, there has been progress in terms of the value of the constant factors involved for special graph classes.
We highlight some of these bounds in Table~\ref{table:bounds}.

{
    \renewcommand{\arraystretch}{1.3}
    \begin{table}[htb!]
        \caption{Known bounds on $c_2(F)$.}\label{table:bounds}
        \begin{tabular}{| l || l || l |}
            \hline
            $F$ & $c_2(F)$ & References\\\hline\hline
            forest & $0$ & \\\hline
            not forest & $\geq 1/2$ & Bollob\'as \cite{BollobasIndependence}
            \\\hline
            \multirow{3}{*}{$K_3$} & finite & Johansson \cite{Joh_triangle} \\\cline{2-3}
            & $\leq 4$ & Pettie--Su \cite{PS15} \\\cline{2-3}
            & $\leq 1$ & Molloy \cite{Molloy}\\\hline
            cycle & $\leq 1$ & Davies--Kang--Pirot--Sereni \cite{DKPS}\\\hline
            fan & $\leq 1$ & Davies--Kang--Pirot--Sereni \cite{DKPS}\\\hline
            bipartite & $\leq 1$ & Anderson--Bernshteyn--Dhawan \cite{AndersonBernshteynDhawan}\\\hline
            \multirow{3}{*}{$K_{1,t,t}$} & $O(t)$ & Alon--Krivelevich--Sudakov \cite{AKSConjecture} \\\cline{2-3}
            & $\leq t$ & Davies--Kang--Pirot--Sereni \cite{DKPS}\\\cline{2-3}
            & $\leq 4$ & Anderson--Bernshteyn--Dhawan \cite{anderson2025coloring}\\\hline
        \end{tabular}
    \end{table}	
}

As every $n$-vertex graph of average degree $d$ contains an $(n/2)$-vertex induced subgraph of maximum degree at most $2d$, it follows that $c_1(F) \geq 1/(4c_2(F))$.
In particular, all of the results in Table~\ref{table:bounds} immediately imply lower bounds on $c_1(F)$ as well.
Our main results in this paper prove Conjectures~\ref{conj: AEKS}~and~\ref{conj: AKS} for all $3$-colorable graphs $F$ (we defer the formal statements to \S\ref{subsection: main results}).

\begin{theorem}
\label{theorem: general}
    Let $F$ be a $3$-colorable graph. Then, for any $F$-free graph $G$ of maximum degree $\Delta$ and average degree $d$, we have
    \[\alpha(G) \geq (1-o(1))\frac{n\log d}{d}, \qquad \text{and} \qquad \chi(G) = O\left(|V(F)|^6\frac{\Delta}{\log \Delta}\right).\]
    In particular, $c_1(F) \geq 1$ and $c_2(F) = O(|V(F)|^6)$.
\end{theorem}

This represents the first progress on both Conjecture~\ref{conj: AEKS} and Conjecture~\ref{conj: AKS} in more than 25 years. Furthermore, we recover—and substantially strengthen—the celebrated result of Shearer from 1983 for triangle-free graphs ($t = 1$) \cite{shearer1983note}, providing an alternative proof of his bound. For $F \subseteq K_{1, t, t}$, our result improves the lower bound on $c_1(F)$ established by Anderson, Bernshteyn, and
the first-named author of this manuscript 
\cite[Corollary~1.7]{anderson2025coloring}, increasing it from $1/16$ to $1$.\footnote{The referenced result is stated in terms of the maximum degree; however, the bound $1/16$ follows immediately as a corollary.} 
Indeed, our result not only establishes that $c_1(F) \ge 1$ for all graphs $F$ for which Conjecture~\ref{conj: AEKS} was previously known to hold, but also extends this conclusion to every $3$-colorable graph~$F$. Prior to our work, this was known only in the special case $F = K_3$, as proved in Shearer’s seminal 1983 paper~\cite{shearer1983note}. 

We note that the case $F = K_{2,2,2}$ has attracted particular recent interest: Ross~Kang posed this specific case as an open problem during the 2024 session of the Sparse Graphs Coalition~\cite{sparsegraphs2024}. Our result (Theorem~\ref{theorem: general}) therefore resolves this question as well.

\subsection{Main results}\label{subsection: main results}

In this section, we include the formal statements of our main results.
We begin with our result on the chromatic number.

\begin{theorem}\label{theorem: coloring}
    For all $t \in \N$, there exists $\Delta_0 \in \N$ such that the following holds for all $\Delta \geq \Delta_0$.
    Let $G$ be a $K_{t, t, t}$-free graph of maximum degree $\Delta$.
    Then
    \[\chi(G) = O\left(\frac{t^6\Delta}{\log \Delta}\right).\]
\end{theorem}

Note that $K_{t,t,t}$-free graphs are not necessarily locally sparse; indeed, the neighborhood of a vertex may even induce a complete bipartite graph. This prevents us from applying the aforementioned result of Alon, Krivelevich, and Sudakov on locally sparse graphs. To establish Theorem~\ref{theorem: coloring}, we instead show that every $K_{t,t,t}$-free graph with maximum degree $\Delta$ admits a vertex ordering in which the \emph{left-neighborhood} (that is, the set of neighbors preceding a vertex in the ordering) is sparse. We discuss this in more detail in the proof overview \S\ref{subsection: proof overview}. We remark that our method in fact yields a more general statement: if a graph~$F$ satisfies Conjecture~\ref{conj: AKS}, then any blow-up of~$F$ also satisfies the conjecture.

For independent sets, we obtain the following result:

\begin{theorem}\label{theorem: iset}
    For all $\epsilon > 0$ and $t \in \N$, there exists $d_0 \in \N$ such that the following holds for all $d\geq d_0$ and $n \in \N$.
    Any $n$-vertex $K_{t, t, t}$-free graph $G$ of average degree $d$ satisfies
    \[\alpha(G) \geq \left(1 - \epsilon\right)\dfrac{n\log d}{d}.\]
\end{theorem}

Note that Theorem~\ref{theorem: coloring} implies $c_1(K_{t, t, t}) = \Omega(1/t^6)$, which is much weaker than the above result.
Additionally, we remark that local sparsity of the left-neighborhood yields the bound $c_1(K_{t, t, t}) = \Omega(1/t^2)$ by employing a recent result of Martinsson and Steiner on the fractional chromatic number of such graphs~\cite{martinsson2025random}.\footnote{See the discussion following the statement of Theorem 1.7 in \cite{martinsson2025random}.}
Again, this bound is much weaker than that in Theorem~\ref{theorem: iset}.
Much of our effort in this paper is in showing $c_1(K_{t, t, t}) \geq 1$, which we believe is optimal with respect to current approaches (see \S\ref{subsection: prior work}). 

Our proof of Theorem~\ref{theorem: iset} combines a new variant of the R\"odl nibble method with a Tur\'an-type bound on the number of triangles in $K_{t,t,t}$-free graphs of bounded maximum degree. The procedure alternates between two complementary phases: a \emph{cleaning} step and a \emph{nibble} step. In each cleaning step, we remove a vertex of high degree, carefully tracking how this removal influences the eventual size of the independent set. Once the degrees in the remaining graph are sufficiently well controlled—that is, when the maximum degree is extremely close to the average degree—we perform a nibble step, extracting a small independent set. We prove that if many cleaning steps are taken during this procedure, even a simple greedy algorithm on the remaining graph produces an independent set of the desired size, whereas if many nibble steps are taken, the union of the extracted sets attains the desired bound.

This tight control on the maximum degree in the remaining graph is crucial for our nibble step. Roughly speaking, it ensures that after each nibble step, the average degree decreases at approximately the same rate as the number of vertices. Without such control, the vertex set could shrink too rapidly, preventing us from performing enough nibble steps to extract an independent set of the desired size. A more detailed discussion of this intuition is provided in the proof overview in~\S\ref{subsection: proof overview}.

As every $3$-colorable graph $F$ is a subgraph of $K_{t, t, t}$ for $t = |V(F)|$, Theorem~\ref{theorem: general} follows as a corollary to Theorems~\ref{theorem: coloring}~and~\ref{theorem: iset}.
We remark that while we treat $t$ as a constant in this paper, our arguments in the proof of Theorem~\ref{theorem: coloring} hold for all $t \leq \alpha_1 \sqrt{\log \Delta/\log \log \Delta}$, and those in the proof of Theorem~\ref{theorem: iset} hold for all $t \leq \alpha_2 \sqrt{\eps\log d/\log \log d}$ for some absolute constants $\alpha_1, \alpha_2 > 0$.
This falls into a line of work concerning $F$-free graphs where $|V(F)| = \omega(1)$ (see, e.g., \cite{anderson2025coloring, AndersonBernshteynDhawan, bonamy2022bounding, anderson2024coloring, dhawan2025bounds, DKPS}).

\subsection{Notation}
Throughout the rest of the paper we use the following basic notation. For $n \in \N$, we let $[n] \defeq \set{1, \ldots, n}$.
For a graph $G$, its vertex and edge sets are denoted $V(G)$ and $E(G)$, respectively.

For a vertex $v \in V(G)$, $N_G(v)$ denotes the neighbors of $v$ and $\deg_G(v) \defeq |N_G(v)|$ denotes the degree of $v$. We let $\Delta(G)$ and $d(G)$ denote the maximum and average degree of $G$, respectively.
Let $N_G[v] \defeq N_G(v) \cup \set{v}$ denote the closed neighborhood of $v$.
For a subset $U \subseteq V(G)$, the subgraph induced by $U$ is denoted by $G[U]$, and $N_G(U)$ is the set of all vertices adjacent to a vertex in $U$, i.e., $N_G(U) \defeq \bigcup_{u \in U}N_G(u)$.
We also let $N_G[U] \defeq U \cup N_G(U)$ and $G - U \defeq G[V(G) \setminus U]$.
We drop the subscript $G$ when the context is clear. For convenience, we write $x=(1\pm\beta)y$ if $(1-\beta)y\le x\le(1+\beta)y$.

\subsection{Comparison with prior work}\label{subsection: prior work}

In this section, we compare our proof method with earlier approaches in related work. 

\subsubsection*{Shearer-style induction}

In his seminal 1983 paper~\cite{shearer1983note}, Shearer proved that every triangle-free graph $G$ of average degree $d$ satisfies $\alpha(G) \ge nf(d)$, where $f(d) = (1-o(1))\frac{\log d}{d}$ is a carefully chosen function. 
The proof proceeds by removing the closed neighborhood of a uniformly random vertex $v \in V(G)$ and considering the resulting graph $G'$; by induction one obtains $\alpha(G) \ge 1 + |V(G')| f(d(G'))$. A key aspect of the argument is to show that, for triangle-free graphs, the expectation of the right-hand side is at least $nf(d)$; the function $f$ is selected as the solution to a differential equation that naturally arises in this computation.

Extending this method even to the $K_{1,t,t}$-free setting encounters several substantial obstacles and has not been achieved in previous work. In particular, it is unclear what is the right choice for $f$ in this case.

A related Shearer-style inductive approach was recently used by Martinsson and Steiner in their work on local fractional colorings of triangle-free graphs, where instead of choosing a uniformly random vertex, they select a vertex $v$ with probability proportional to a weight assigned to each vertex~\cite{martinsson2025random}.
Buys, van den Heuvel, and Kang refer to this method as ``induction \'a la Shearer'' in their study of counting colorings of triangle-free graphs~\cite{buys2025triangle}. Their work provides a strict generalization of this technique, yielding a lower bound on the partition function of triangle-free graphs arising from the hard-core model; see~\cite{davies2025hard} for an introduction to this model and its applications to graph coloring.
We remark that the difficulty of extending these approaches to $K_{1,t,t}$-free graphs persists across all such applications.

\subsubsection*{Random independent sets}

In Shearer's follow-up work on $K_r$-free graphs of maximum degree $\Delta$ \cite{shearer1995independence}, he proves the bound $\alpha(G) = \Omega\left(\frac{n\log \Delta}{r\Delta\log\log \Delta}\right)$ by showing that a uniformly random independent set $I$ has the desired size in expectation.
His proof proceeds as follows: define the random variables
\[X_v = \Delta\,\bbone{v\in I} + |N(v) \cap I|, \qquad \text{and} \qquad X = \sum_v X_v.\] 
For any independent set $J$ disjoint from $N[v]$, he shows that
\[\E[X_v \mid I \setminus N[v] = J] = \Omega\left(\frac{\log \Delta}{r\log\log \Delta}\right).\]

As the right hand side above is independent of $v$ and $J$, one combines it with the fact that $X \leq 2\Delta|I|$ to complete the proof.
This approach is now referred to as a `Shearer-type argument' and has been applied to a number of graph classes satisfying various local structural constraints (see, e.g., \cite{DKPS, dhawan2025bounds, alon1996independence, davies2018average}). Unfortunately, this approach does not directly yield the bound $\alpha(G)=\Omega\left(\frac{n\log \Delta}{\Delta}\right)$ for $K_{t,t,t}$-free graphs $G$ with $t \ge 2$, as we encounter the same difficulty as in the proof for $K_r$-free graphs; therefore, we must employ a different method to establish our main results.

Furthermore, the approach does not work in the average degree setting.
As mentioned earlier, a key lemma in our proof shows that $G$ contains ``few'' triangles globally.
This implies that $G$ contains a locally sparse induced subgraph on $(1 - o_\Delta(1))n$ vertices.\footnote{To see this, let $G$ be a graph with maximum degree $\Delta$ and at most $n\Delta^{2-\eps}$ triangles. Then the subgraph induced by vertices contained in at most $\Delta^{2-\eps/2}$ triangles still has maximum degree at most $\Delta$, is $\Delta^{\eps/2}$-locally sparse, and contains at least $(1 - \Delta^{-\eps/2})n$ vertices.}
With this in hand, one can prove the bound $c_1(K_{t, t, t}) \geq 1/(4t^2)$ by employing a Shearer-type argument, which is weaker than Theorem~\ref{theorem: iset}.
The proof follows, for instance, by applying \cite[Theorem~23~(ii)]{DKPS} to the locally sparse induced subgraph.

\subsubsection*{The nibble method}

Another technique that is heavily employed in results on independent sets and colorings of (hyper)graphs is the `R\"odl nibble method,' an iterative procedure to construct the desired structure.
With regards to independent sets, the procedure proceeds as follows: randomly construct a ``small'' independent set $I$ in $G$ and repeat the same procedure on $G' \defeq G- N[I]$.
When considering graphs of girth at least five, one can show that $\Delta(G')$ is concentrated around its expected value, leading to a proof of the bound\footnote{This was first observed by Kim in the coloring setting \cite{Kim95}; see \cite[Ch. 12]{MolloyReed} for a textbook treatment of the argument.} $\alpha(G) \geq (1-o(1))\frac{n\log\Delta}{\Delta}$.
Recently, Campos, Jenssen, Michelen, and Sahasrabudhe developed a variant of this method for graphs with bounded codegree and applied it to problems on sphere packings~\cite{campos2023new}. The chromatic number of such graphs was later studied in~\cite{bradshaw2025toward} using a different variant of the nibble method, developed by Bradshaw, 
Wigal, 
and the first and third named authors of the present manuscript,
in connection with Vu's conjecture.

When considering triangle-free graphs, however, $\Delta(G')$ is no longer concentrated.
Indeed, the possible presence of many $4$-cycles precludes the use of Azuma-type concentration inequalities entirely.
To overcome this obstacle, one instead concentrates the average degree. 
This idea has been employed extensively in the graph coloring literature (cf.~\cite{Joh_triangle, PS15, anderson2024coloring, anderson2025coloring, dhawan2024palette, Jamall}). 
However, the argument still requires control over $\Delta(G')$, necessitating a cleaning step in which high-degree vertices are removed; see, e.g., the proof overview in~\S\ref{subsection: proof overview}. 
The key component of our analysis is a novel cleaning procedure that carefully tracks how the removal of high-degree vertices affects the eventual size of the independent set. 
This allows us to obtain an effective bound in terms of the average degree (see~\S\ref{subsection: proof overview} for further discussion).

\subsubsection*{Shattering threshold}

We conclude this section with a discussion of the optimality of $c_1(K_{t,t,t})$ with respect to our proof technique.
It is known that there exist $n$-vertex $d$-regular graphs of arbitrarily high girth whose independence number satisfies $\alpha(G) \le (2+o(1))\frac{n\log d}{d}$~\cite{BollobasIndependence}.
However, proving that a pseudorandom graph class $\mathcal{G}$ satisfies $\alpha(G) > \frac{n\log d(G)}{d(G)}$ for all $G \in \mathcal{G}$ remains a tantalizing open problem.
The quantity $\frac{n\log d(G)}{d(G)}$ appears to be a natural threshold for lower bounds on $\alpha(G)$, as it coincides with a well-known algorithmic barrier—often referred to as the shattering threshold~\cite{Achlioptas, coja2015independent}—for independent sets in random graphs of average degree $d$.
A celebrated conjecture of Karp~\cite{karp1976probabilistic} asserts that no polynomial-time algorithm can, with high probability, construct an independent set of size at least $(1+\eps)\frac{n\log d}{d}$ for any $\eps>0$ in the sparse random graph $G(n, d/n)$.
Proving this conjecture would imply a statement even stronger than $P \ne NP$, and hence researchers have instead focused on providing evidence in its favor: for several restricted classes of algorithms, this bound is indeed optimal~\cite{RV, wein2020optimal, gamarnik2025optimal}.
Since $G(n, d/n)$ is $K_{t,t,t}$-free with high probability (as long as $n$ is sufficiently large compared to $d$), any efficient algorithmic proof of the bound $\alpha(G) \ge (1+\eps)\frac{n\log d}{d}$ for $n$-vertex $K_{t,t,t}$-free graphs of average degree $d$ would in fact disprove Karp’s conjecture.
As we will see, our proof of Theorem~\ref{theorem: iset} yields a polynomial-time algorithm for constructing such independent sets.
In fact, both Shearer-style induction and the nibble method lead to polynomial-time constructions as well.
Moreover, in the hard-core model refinement of the Shearer-type random independent set approach, efficient sampling algorithms are known to exist below a certain critical regime closely related to the shattering threshold~\cite{bhatnagar2016decay, weitz2006counting, galvin2006slow}.
In light of Karp’s conjecture, we therefore believe that proving $c_1(K_{t,t,t}) > 1$ would require a substantially new idea—one that departs significantly from all existing approaches in the literature, even for the case $t=1$.

\subsection{Proof overview}\label{subsection: proof overview} In this section, we provide a detailed sketch of the main ideas behind our proofs of Theorems~\ref{theorem: coloring} and~\ref{theorem: iset}.

\subsubsection*{Left local sparsity and the proof of Theorem~\ref{theorem: coloring}}

While a $K_{t,t,t}$-free graph with bounded maximum degree need not be locally sparse—that is, individual vertices may still lie in many triangles—we show that such graphs have a small global triangle count. Specifically, we show that every $K_{t,t,t}$-free graph with maximum degree $\Delta$ contains at most $O(n\Delta^{2 - 1/t^2})$ triangles.
This global bound plays a key role in the proofs of both Theorem~\ref{theorem: coloring} and Theorem~\ref{theorem: iset}. In the proof of Theorem~\ref{theorem: coloring}, we convert this global sparsity into a form of local sparsity by identifying a suitable vertex ordering: we show that there exists an ordering of the vertex set in which each vertex $u$ appears in few triangles where $u$ is the rightmost vertex. In other words, while a vertex’s full neighbourhood may not be sparse, its left-neighbourhood in the ordering is.

Recall that the result of Alon, Krivelevich, and Sudakov~\cite{AKSConjecture} bounds the chromatic number of any graph with given maximum degree and sparse neighbourhoods. By adapting their argument, we extend this bound to graphs that admit a vertex ordering in which each vertex has a sparse left-neighbourhood. Using our bound on the number of triangles, we show that all $K_{t,t,t}$-free graphs with bounded maximum degree admit such a vertex ordering, thereby yielding the desired bound on the chromatic number.

\subsubsection*{The R\"odl nibble for the proof of Theorem~\ref{theorem: iset}}

As mentioned earlier, we develop a new variant of the R\"odl nibble method to prove Theorem~\ref{theorem: iset}. Classically, the nibble method for constructing large independent sets proceeds as follows:
\begin{enumerate} [label=\textup{(G\arabic*)}, leftmargin=50pt]
   \item\label{G1} Initialize $G_0 = G$ and set $i = 0$.

\item\label{G2} Let $A_i$ be a $p_i$-random subset of $V(G_i)$, where $p_i \coloneqq \frac{\eps}{\Delta(G_i)}$.

\item\label{G3} Let $V_i' \coloneqq V(G_i) \setminus N[A_i]$, and let $I_i \subseteq A_i$ be a maximum independent set in $G_i[A_i]$.

\item\label{G4} Set $G_{i+1} \coloneqq G_i[V_i']$ and repeat steps~\ref{G2}--\ref{G4}.
\end{enumerate}
We continue iterating steps~\ref{G2}--\ref{G4} as long as
\begin{align}\label{eq: max deg termination cond}
    |V(G_i)| \geq \frac{|V(G)|}{\Delta(G)^{1 - \Theta(\eps)}}, \qquad \text{and} \qquad \Delta(G_i) \geq \Delta(G)^{\Theta(\eps)}.
\end{align}

For graphs of girth at least $5$, after each iteration of the above procedure both the maximum degree $\Delta(G_i)$ and the number of vertices $|V(G_i)|$ decrease by a factor of roughly $e^{-\eps}$. Consequently, condition~\eqref{eq: max deg termination cond} holds for all $i \le (1 - \eps)\frac{\log \Delta(G)}{\eps}$. Moreover, the following estimates hold for every $i \le (1 - \eps)\frac{\log \Delta(G)}{\eps}$: 
\[
|I_i| = \alpha(G_i[A_i]) \gtrsim \frac{\eps |V(G_i)|}{\Delta(G_i)} 
\qquad \text{and} \qquad 
\frac{|V(G_i)|}{\Delta(G_i)} \ge (1 - \eps)\frac{|V(G)|}{\Delta(G)}.
\]
Therefore, the union $\bigcup_i I_i$ forms an independent set in $G$ of size at least $(1 - 2\eps)\frac{|V(G)| \log \Delta(G)}{\Delta(G)}$.


When dealing with triangle-free graphs, the maximum degree $\Delta(G_i)$ is no longer necessarily concentrated throughout the iterations, as the possible presence of numerous $4$-cycles prevents the application of any Azuma-type concentration inequality. However, the average degree $d(G_i)$ does remain concentrated; this was first observed by Ajtai, Koml\'os, and Szemer\'edi in~\cite{ajtai1980note} (see also~\cite{PS15, alon2020palette, anderson2025coloring, Jamall} for coloring variants of the argument). Unfortunately, some control over the maximum degree is still required to ensure that condition~\eqref{eq: max deg termination cond} holds for sufficiently many iterations. This necessitates the introduction of a cleaning step (see step~\ref{step: cleaning triangle} below).
\begin{enumerate}[label=\textup{(AKS\arabic*)}, leftmargin=50pt]
    \item\label{AKS1} Initialize $G_0 = G$ and set $i = 0$.
    
    \item\label{AKS2} Let $A_i$ be a $p_i$-random subset of $V(G_i)$, where $p_i \coloneqq \frac{\eps}{\Delta(G_i)}$.
    
    \item\label{AKS3} Let $V_i' \coloneqq V(G_i) \setminus N[A_i]$, and let $I_i \subseteq A_i$ be a maximum independent set in $G_i[A_i]$.
    
    \item\label{step: cleaning triangle} Let $V_i'' \subseteq V_i'$ be the set of vertices of degree at most $2d(G_i')$ in the subgraph $G_i' \coloneqq G_i[V_i']$.
    
    \item\label{AKS5} Set $G_{i+1} \coloneqq G_i[V_i'']$ and repeat steps~\ref{AKS2}--\ref{AKS5}.
\end{enumerate}
The cleaning step results in a constant-factor loss of $4$. Roughly speaking, we lose a factor of $2$ due to the imposed upper bound on $\Delta(G_i)$, and potentially another factor of $2$ because at most half of the vertices may be discarded when passing from $V_i'$ to $V_i''$. More formally, one can show that condition~\eqref{eq: max deg termination cond} holds only for all $i \le (1 - \eps)\frac{\log \Delta(G)}{2\eps}$ (rather than $i \le (1 - \eps)\frac{\log \Delta(G)}{\eps}$). Moreover, we have
\[
|I_i| = \alpha(G_i[A_i]) \gtrsim \frac{\eps |V(G_i)|}{\Delta(G_i)}
\qquad \text{and} \qquad
\frac{|V(G_i)|}{\Delta(G_i)} \ge (1 - \eps)\frac{|V(G)|}{2\Delta(G)}.
\]
Therefore, the union $\bigcup_i I_i$ forms an independent set in $G$ of size at least $(1 - 2\eps)\frac{|V(G)| \log \Delta(G)}{4\Delta(G)}$. 

The triangle-freeness of $G$ is used to ensure the concentration of $d(G_i)$ during each iteration of the nibble. In fact, for this concentration to hold, it suffices that $G_i$ contain only ``few'' triangles.\footnote{More precisely, we require that $G_i$ contain at most $|V(G_i)|\Delta(G_i)^{2  - \delta}$ triangles, where $\delta = \omega\left(\frac{\log\log\Delta(G_i)}{\log\Delta(G_i)}\right)$.}
By combining the above nibble approach with our Tur\'an-type bound on the number of triangles in $K_{t,t,t}$-free graphs of a given maximum degree, we can easily obtain

\begin{equation}
\label{lossoffactor4}
\alpha(G) \ge (1 - \eps)\frac{|V(G)| \log \Delta(G)}{4\Delta(G)}.
\end{equation}

The main novelty of our result (Theorem~\ref{theorem: iset}) is that it removes the constant-factor loss of $4$ in~\eqref{lossoffactor4} for $K_{t,t,t}$-free graphs, thereby matching the shattering threshold and Shearer's celebrated bound for triangle-free graphs. Furthermore, our proof is notably simpler, avoiding many of the technical computations required in previous approaches.

Key to our proof is a novel cleaning procedure that not only improves the constant factor from $1/4$ to $1$, but also extends the result to a formulation in terms of the \emph{average degree} $d(G)$ rather than the maximum degree $\Delta(G)$, for all $K_{t,t,t}$-free graphs $G$. As noted earlier, such a bound in terms of the average degree was previously known only for $K_3$-free graphs by Shearer’s seminal work~\cite{shearer1983note}. 

An outline of our nibble method is as follows: 
\begin{enumerate}[label=\textup{(DJM\arabic*)}, leftmargin=50pt]
    \item \label{DJM1} Initialize $G_0 = G$ and set $i = 0$.
    \item \label{DJM2} If there exists a vertex of degree greater than $(1+\eps)d(G_i)$, choose such a vertex $v$. Set $G_{i+1} \coloneqq G_i - v$, keep track of how this removal affects the eventual size of the independent set, and repeat steps~\ref{DJM2}--\ref{DJM3}.
    \item \label{DJM3} If $\Delta(G_i) \le (1+\eps)d(G_i)$, then:
    \begin{enumerate}
        \item Let $A_i$ be a $p_i$-random subset of $V(G_i)$, where $p_i \coloneqq \frac{\eps}{d(G_i)}$.
        \item Let $V_i' \coloneqq V(G_i) \setminus N[A_i]$, and let $I_i \subseteq A_i$ be a maximum independent set in $G_i[A_i]$.
        \item Set $G_{i+1} \coloneqq G_i[V_i']$ and repeat steps~\ref{DJM2}--\ref{DJM3}.\footnote{We note that in the actual algorithm, $G_{i+1} \subseteq G_i[V_i']$ is the induced subgraph on the set of vertices in $V_i'$ that survive their so-called \textit{equalizing coin flips}. While we omit this technical detail in the current overview, we emphasize that the coin flip mechanism is also one of the reasons we require control over $\Delta(G_i)$.}
    \end{enumerate}
\end{enumerate}

In particular, our nibble procedure alternates between \emph{cleaning} and \emph{nibble} steps.
During a cleaning step, we remove a single vertex of high degree, while during a nibble step we extract an independent set of size roughly $\eps \frac{|V(G_i)|}{d(G_i)}$.

To streamline the remainder of this proof overview, we introduce the parameters
\[
N_i \coloneqq |V(G_i)|, \qquad \Delta_i \coloneqq \Delta(G_i), \qquad D_i \coloneqq d(G_i).
\]
Assume that the procedure runs for $T$ iterations, and let $S \subseteq [T]$ denote the set of iterations in which a nibble step is performed.

A key observation in our proof is that although both $N_i$ and $D_i$ decrease during a cleaning step, their ratio $N_i / D_i$ increases by a factor of approximately $\left(\frac{N_i}{N_i - 1}\right)^{\eps}$ (see Corollary~\ref{cor:one cleaning step ratio}). 
To quantify this effect, we define
\[
R_i \coloneqq \prod_{j < i,\, j \notin S} \frac{N_j}{N_j - 1}.
\]
The term $R_i^{\eps}$ thus represents the cumulative gain in the ratio $N_i / D_i$ contributed by all cleaning steps up to iteration~$i$.

Intuitively, our argument proceeds as follows: if $|S|$ is small and $D_T$ remains sufficiently large (that is, most steps are cleaning steps), then the gain in $N_i / D_i$ is large enough to guarantee that a greedy independent set in $G_T$ already has the desired size; 
conversely, if $|S|$ is sufficiently large (that is, enough nibble steps are taken), then $\bigcup_i I_i$ itself has the desired size.
With this intuition in place, we now state the termination condition of the nibble procedure more precisely. We stop the process if any of the following holds:
\begin{align}\label{eq: termination avg deg}
    D_T < D_1^{\Theta(\eps)}, \qquad 
    R_T \ge D_1^{\Theta(\eps)}, \qquad \text{or} \qquad 
    |S| = (1 - \eps)\frac{\log D_1}{\eps}.
\end{align}

Indeed, given this termination condition, we show that the following inequalities hold for all $i \in [T]$ (see Lemma~\ref{lemma: lb on n/d and n}):
\begin{align}\label{eq: lemma on n/d and n}
    \frac{N_i}{D_i} \geq (1-\eps)R_i^{\eps}\,\frac{N_1}{D_1}, \qquad \text{and} \qquad N_i \geq \frac{N_1}{D_1^{1 - \Theta(\eps)}}.
\end{align}
With this in hand, we prove Theorem~\ref{theorem: iset} as follows (see \S\ref{sec:propassuminglem} for a detailed computation):
\begin{itemize}
    \item If the process terminates due to the condition on $D_T$ in~\eqref{eq: termination avg deg}, then a greedy independent set in $G_T$ has the desired size by combining $D_T < D_1^{\Theta(\eps)}$ with the lower bound on $N_T$ given in~\eqref{eq: lemma on n/d and n}.
    \item If the process terminates due to the condition on $R_T$ in~\eqref{eq: termination avg deg}, then a greedy independent set in $G_T$ has the desired size by combining $R_T \ge D_1^{\Theta(\eps)}$ with the lower bound on $N_T / D_T$ in~\eqref{eq: lemma on n/d and n}.
    \item If the process terminates due to the condition on $|S|$ in~\eqref{eq: termination avg deg}, then recalling that $|I_i| \ge \eps N_i / D_i$ for all $i \in [T]$, it follows that $\bigcup_i I_i$ has the desired size by combining $|S| = (1 - \eps)\frac{\log D_1}{\eps}$ with the lower bound on $N_i / D_i$ in~\eqref{eq: lemma on n/d and n} and the fact that $R_i \ge 1$ for all~$i \in [T]$.
\end{itemize}

We emphasize that maintaining a tight upper bound on~$\Delta_i$ at every nibble step~\ref{DJM3} is critical to the success of our method. In the remainder of this overview, we explain why this degree control is necessary. During the $i$-th nibble step, we consider an $N_i$-vertex graph with average degree $D_i$ and maximum degree $\Delta_i$. We then find an independent set $I_i \subseteq A_i$ of size roughly $\eps N_i / D_i$, and construct a subgraph $G_{i+1} \subseteq G_i - N[A_i]$ satisfying $N_{i+1} \gtrsim \gamma_i N_i$ and $D_{i+1} \lesssim \gamma_i D_i$, where $\gamma_i \approx \exp(-\eps \Delta_i / D_i)$.

Suppose that only a few cleaning steps are performed throughout the nibble procedure, so that $R_i \approx 1$ for all $i \in [T]$ and $D_T$ remains sufficiently large. Then, by~\eqref{eq: lemma on n/d and n},
\[
\frac{N_i}{D_i} \gtrsim (1 - \eps)\frac{N_1}{D_1} = (1 - \eps)\frac{|V(G)|}{d(G)} 
    \qquad \text{for all } i \in [T].
\]
It follows that
\[
\Big|\bigcup_i I_i\Big| \gtrsim \eps \sum_i \frac{N_i}{D_i}
    \gtrsim (1 - \eps)\eps |S| \frac{|V(G)|}{d(G)}.
\]
Hence, to guarantee the desired bound in Theorem~\ref{theorem: iset}, we require $|S| \gtrsim \frac{\log d(G)}{\eps}$, that is, we need sufficiently many nibble steps to be performed.

To estimate how many such steps can be performed, observe that if $\Delta_i \le C D_i$ for some constant $C > 0$, then $\gamma_i \approx \exp(-\eps \Delta_i / D_i) \gtrsim \exp(-\eps C)$, and therefore
\[
N_T \gtrsim N_1 \exp\!\left(-\eps \sum_{i \in S} \frac{\Delta_i}{D_i}\right)
    \gtrsim N_1 \exp(-C\eps |S|).
\]
Consequently, the nibble procedure can be executed for at most 
$|S| = (1 - \eps)\frac{\log D_1}{C\eps} 
    = (1 - \eps)\frac{\log d(G)}{C\eps}$
steps while maintaining the lower bound $N_T \ge N_1 / D_1^{1 - \Theta(\eps)}$ in~\eqref{eq: lemma on n/d and n}. This yields
\[
\Big|\bigcup_i I_i\Big| 
    \ge (1 - \eps)\eps |S| \frac{|V(G)|}{d(G)} 
    \ge \frac{(1 - 2\eps)}{C}\frac{|V(G)|\log d(G)}{d(G)}.
\]
In particular, to achieve the bound in Theorem~\ref{theorem: iset}, we must take $C$ as small as possible; since $C \ge 1$, this forces us to set $C = 1 + \varepsilon$, which explains why degree control is necessary.

\subsection{Structure of the paper}
The rest of the paper is organized as follows.
In \S\ref{section: prelim}, we collect the probabilistic tools we need.
In \S\ref{section: ktttfree}, we prove a Tur\'an-type result for $K_{t, t, t}$-free graphs with bounded maximum degree and use it to show that such graphs admit a vertex ordering with sparse left-neighborhoods.
In \S\ref{section: coloring}, we prove Theorem~\ref{theorem: coloring}.
In \S\ref{section: ISET proof}, we establish Theorem~\ref{theorem: iset}, assuming a key result, Proposition~\ref{prop: alg iset}, concerning the extraction of independent sets in the nibble step; the proof of this proposition is given in \S\ref{section: proof of prop}.
Finally, in \S\ref{section: concluding remarks}, we discuss some open problems.

\section{Preliminaries}\label{section: prelim}

In this section we outline the main background facts that will be used in our arguments. We start with the symmetric version of the Lov\'asz Local Lemma.

\begin{theorem}[{Lov\'asz Local Lemma; \cite[\S4]{MolloyReed}}]\label{theorem: LLL}
    Let $A_1$, $A_2$, \ldots, $A_n$ be events in a probability space. Suppose there exists $p \in [0, 1)$ such that for all $1 \leq i \leq n$ we have $\Pr[A_i] \leq p$. Further suppose that each $A_i$ is mutually independent from all but at most $d_{LLL}$ other events $A_j$, $j\neq i$ for some $d_{LLL} \in \N$. If $4pd_{LLL} \leq 1$, then with positive probability none of the events $A_1$, \ldots, $A_n$ occur.
\end{theorem}

In addition to the Local Lemma, we will use a few standard concentration inequalities. The first is the Chernoff bound for binomial random variables, stated below in its two-tailed form:

\begin{theorem}[{Chernoff Bound; \cite[\S5]{MolloyReed}}]\label{theorem: chernoff}
    Let $X$ be a binomial random variable on $n$ independent trials, each with probability $p$ of success. Then for any $0 \leq t \leq \E[X]$, we have
    \begin{align*}
        \Pr\Big[\big|X - \E[X]\big| \geq t\Big] < 2\exp{\left(-\frac{t^2}{3\E[X]}\right)}. 
    \end{align*}
\end{theorem}

We will also take advantage of Talagrand's inequality.

\begin{theorem}[{Talagrand's Inequality; \cite{molloy2014colouring}}]\label{theorem: Talagrand}
    Let $X$ be a random variable taking values in the non-negative integers, not identically $0$, and defined as a function of $n$ independent trials $T_1, \ldots, T_n$. Suppose that $X$ satisfies the following for some $\mu$, $r > 0$: 
    \begin{enumerate}[label=\textup{(T\arabic*)}]
        \item Changing the outcome of any one trial $T_i$ can change $X$ by at most $\mu$.
        \item For any integer $s \geq 1$, if $X \ge s$, then there exists a set of at most $rs$ trials that certifies $X \ge s$.
    \end{enumerate}
    Then for any $t \geq 0$, we have
    \begin{align*}
        \Pr\Big[\big|X-\E[X]\big| \geq t + 20\mu\sqrt{r\E[X]} + 64\mu^2r\Big]\leq 4\exp{\left(-\frac{t^2}{8\mu^2r(\E[X] + t)}\right)}.
    \end{align*}
\end{theorem}

\section{Tur\'an-type results for $K_{t,t,t}$-free graphs}\label{section: ktttfree}

In this section, we establish two key results about $K_{t,t,t}$-free graphs, as outlined in \S\ref{subsection: proof overview}: 
\begin{enumerate}
\item Every $K_{t,t,t}$-free graph with $n$ vertices and maximum degree $\Delta$ contains at most $Cn\Delta^{2 - 1/t^2}$ triangles for some $C \coloneqq C(t)$; and
\item Every $K_{t,t,t}$-free graph with maximum degree $\Delta$ admits a vertex ordering in which each vertex is the rightmost vertex in at most $Cn\Delta^{2 - 1/t^2}$ triangles for some $C \coloneqq C(t)$.
\end{enumerate}
We begin with the first result.

\begin{lemma} \label{lemma:fewtriangles}
    For every positive integer $t$, there is a constant $C\coloneqq C(t)$ such that the following holds. Let $G$ be a $K_{t,t,t}$-free graph on $n$ vertices with maximum degree at most $\Delta$. Then $G$ has at most $Cn\Delta^{2-1/t^2}$ triangles.
\end{lemma}

\begin{proof}\stepcounter{ForClaims}\renewcommand{\theForClaims}{\ref*{lemma:fewtriangles}}
    We may assume that $t\geq 2$, otherwise the statement is trivial. Let $\eps \coloneqq 1/t^2$, $C \coloneqq C(t)$ be sufficiently large, and let $G$ be a $K_{t,t,t}$-free graph on $n$ vertices with maximum degree at most $\Delta$. Suppose that $G$ contains more than $Cn\Delta^{2-\eps}$ triangles.

    \begin{claim} \label{claim:k11t}
        $G$ contains at least $Cn\Delta^{t+1-t\eps}$ labelled copies of $K_{1,1,t}$.
    \end{claim}

    \begin{claimproof}
        For each labelled edge $e$, let $q_e$ be the number of common neighbours in $G$ of the vertices of $e$ (see Fig.~\ref{fig: K_11t} for an illustration). Observe that the number of labelled triangles in $G$ is precisely $\sum_{e} q_e$ and the number of labelled copies of $K_{1,1,t}$ is at least $\sum_{e} \binom{q_e}{t}$. In particular, by assumption, we have $\sum_{e} q_e>Cn\Delta^{2-\eps}$.

        \begin{figure}[htb!]
            \centering
            \begin{tikzpicture}
                \node[circle,fill=black,draw,inner sep=0pt,minimum size=4pt] (u) at (0,0) {};
                \node[circle,fill=black,draw,inner sep=0pt,minimum size=4pt] (v) at (1,0) {};
    
                \node[circle,fill=black,draw,inner sep=0pt,minimum size=4pt] (x1) at (-1.5,2) {};
                \node[circle,fill=black,draw,inner sep=0pt,minimum size=4pt] (x2) at (-0.5,2) {};
                \node[circle,inner sep=0pt,minimum size=4pt] at (0.5, 2) {$\ldots$};
                \node[circle,fill=black,draw,inner sep=0pt,minimum size=4pt] (x3) at (1.5,2) {};
                \node[circle,fill=black,draw,inner sep=0pt,minimum size=4pt] (x4) at (2.5,2) {};
    
                \draw[thick] (u) to node[below,inner sep=1pt,outer sep=1pt,minimum size=4pt,fill=white] {$e$} (v);
    
                \draw[thick, dotted, red] (u) -- (x1) -- (v); 
                \draw[thick, dotted, blue] (u) -- (x2) -- (v);
                \draw[thick, dotted, green] (u) -- (x3) -- (v);
                \draw[thick, dotted, orange] (u) -- (x4) -- (v);
    
                \draw[decoration={brace,amplitude=10pt}, decorate] (-1.6, 2.1) -- node [midway,above,xshift=0pt,yshift=10pt] {$q_e$} (2.6,2.1);
            \end{tikzpicture}
            \caption{An edge $e$ contained in $q_e$ copies of $K_3 \cong K_{1, 1, 1}$.}
            \label{fig: K_11t}
        \end{figure}
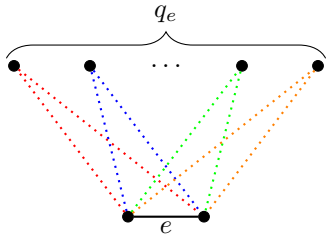

        The number of labelled edges in $G$ is at most $n\Delta$, therefore we have 
        \[\sum_{e:q_e<t} q_e<n\Delta t<\frac{1}{2}Cn\Delta^{2-\eps}.\] 
        Hence,
        \[\sum_{e:q_e\geq t} q_e\geq \frac{1}{2}Cn\Delta^{2-\eps}.\]

        Observe that whenever $q_e\geq t$, we have $\binom{q_e}{t}\geq q_e^t/t^t$. Thus, by Jensen's inequality, and writing $m$ for the number of labelled edges $e$ with $q_e\geq t$, we have
        \begin{align*}
            \sum_{e:q_e\geq t} \binom{q_e}{t}
            &\geq \frac{1}{t^t}\sum_{e:q_e\geq t} q_e^t\geq \frac{1}{t^t}m\left(\frac{\sum_{e: q_e\geq t} q_e}{m}\right)^t\geq \frac{1}{t^t}n\Delta \left(\frac{\sum_{e: q_e\geq t} q_e}{n\Delta }\right)^t \\
            &\geq \frac{1}{t^t}n\Delta  (C\Delta^{1-\eps}/2)^t\geq Cn\Delta^{t+1-t\eps},
        \end{align*}
        where in the last inequality we used the assumption that $C$ is sufficiently large compared to $t$. This completes the proof of Claim \ref{claim:k11t}.
    \end{claimproof}

    \begin{claim} \label{claim:k1tt}
        $G$ contains at least $Cn\Delta^{2t-t^2\eps}$ labelled copies of $K_{1,t,t}$.
    \end{claim}

    \begin{claimproof}
        For each labelled copy $S$ of the star $K_{1,t}$ in $G$, let $q_S$ be the number of common neighbours in $G$ of the vertices of $S$ (see Fig.~\ref{fig: K_1tt} for an illustration). Observe that the number of labelled copies of $K_{1,1,t}$ in $G$ is precisely $\sum_{S} q_S$ and the number of labelled copies of $K_{1,t,t}$ is at least $\sum_{S} \binom{q_S}{t}$. In particular, by Claim \ref{claim:k11t}, we have $\sum_{S} q_S>Cn\Delta^{t+1-t\eps}$.

        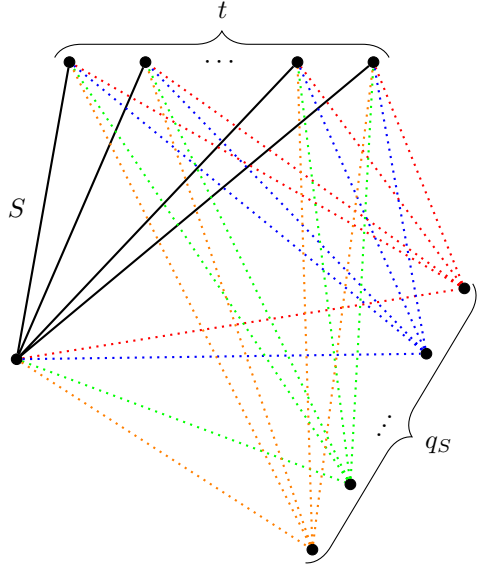
\begin{figure}[htb!]
            \centering
            \begin{tikzpicture}
                
                \node[circle,fill=black,draw,inner sep=0pt,minimum size=4pt] (u) at (0,0) {};
                \path (u) ++(80:4) node[circle,fill=black,draw,inner sep=0pt,minimum size=4pt] (v1) {};
                \path (v1) ++(-30:6) node[circle,fill=black,draw,inner sep=0pt,minimum size=4pt] (x1) {};
    
                \path (v1) ++(0:1) node[circle,fill=black,draw,inner sep=0pt,minimum size=4pt] (v2) {};
                \path (x1) ++(-120:1) node[circle,fill=black,draw,inner sep=0pt,minimum size=4pt] (x2) {};
    
                \path (v2) ++(0:1) node[circle,inner sep=0pt,minimum size=4pt] {$\cdots$};
                \path (x2) ++(-120:1) node[circle,inner sep=0pt,minimum size=4pt, rotate=-30] {$\vdots$};
    
                \path (v2) ++(0:2) node[circle,fill=black,draw,inner sep=0pt,minimum size=4pt] (v3) {};
                \path (x2) ++(-120:2) node[circle,fill=black,draw,inner sep=0pt,minimum size=4pt] (x3) {};
    
                \path (v3) ++(0:1) node[circle,fill=black,draw,inner sep=0pt,minimum size=4pt] (v4) {};
                \path (x3) ++(-120:1) node[circle,fill=black,draw,inner sep=0pt,minimum size=4pt] (x4) {};
    
                \node[circle,inner sep=0pt,minimum size=4pt] at (0, 2) {$S$};
    
                \draw[thick] (u) -- (v1) (u) -- (v2) (u) -- (v3) (u) -- (v4);
    
                \draw[thick, dotted, red] (x1) -- (v1) (u) -- (x1) (x1) -- (v2) (x1) -- (v3) (x1) -- (v4);
                \draw[thick, dotted, blue] (x2) -- (v1) (u) -- (x2) (x2) -- (v2) (x2) -- (v3) (x2) -- (v4);
                \draw[thick, dotted, green] (x3) -- (v1) (u) -- (x3) (x3) -- (v2) (x3) -- (v3) (x3) -- (v4);
                \draw[thick, dotted, orange] (x4) -- (v1) (u) -- (x4) (x4) -- (v2) (x4) -- (v3) (x4) -- (v4);
    
                \draw[decoration={brace,amplitude=10pt}, decorate] (0.5, 4) -- node [midway,right,xshift=-5pt,yshift=18pt] {$t$} (4.9, 4);
                \draw[decoration={brace,amplitude=10pt}, decorate] (6, 1) -- node [midway,right,xshift=10pt,yshift=-8pt] {$q_S$} (3.8, -2.7);
            \end{tikzpicture}
            \caption{A copy $S$ of $K_{1, t}$ contained in $q_S$ copies of $K_{1, 1, t}$.}
            \label{fig: K_1tt}
        \end{figure}

        The number of labelled copies of $K_{1,t}$ in $G$ is at most $n\Delta^t$ since $G$ has maximum degree at most $\Delta$. Therefore, we have 
        \[\sum_{S:q_S<t} q_S<n \Delta^t t<\frac{1}{2}Cn\Delta^{t+1-t\eps}.\]
        Hence,
        \[\sum_{S:q_S\geq t} q_S\geq \frac{1}{2}Cn\Delta^{t+1-t\eps}.\]

        Observe that whenever $q_S\geq t$, we have $\binom{q_S}{t}\geq q_S^t/t^t$. Thus, by Jensen's inequality, and writing $r$ for the number of $t$-leaf labeled stars $S$ with $q_S\geq t$, we have
        \begin{align*}
            \sum_{S:q_S\geq t} \binom{q_S}{t}
            &\geq \frac{1}{t^t}\sum_{S:q_S\geq t} q_S^t\geq \frac{1}{t^t}r\left(\frac{\sum_{S: q_S\geq t} q_S}{r}\right)^t\geq \frac{1}{t^t}n\Delta^t\left(\frac{\sum_{S: q_S\geq t} q_S}{n\Delta^t}\right)^t \\
            &\geq \frac{1}{t^t}n\Delta^t (C\Delta^{1-t\eps}/2)^t\geq Cn\Delta^{2t-t^2\eps}.
        \end{align*}
        This completes the proof of Claim \ref{claim:k1tt}.
    \end{claimproof}

    Now observe that, since the maximum degree of $G$ is at most $\Delta$, the number of labeled copies of $K_{t,t}$ in $G$ is at most $n\Delta^{2t-1}$. Since \[Cn\Delta^{2t-t^2\eps}=Cn\Delta^{2t-1}>n\Delta^{2t-1}\cdot t,\] it follows by Claim \ref{claim:k1tt} that there exists a copy of $K_{t,t}$ which extends to at least $t$ copies of $K_{1,t,t}$. The union of these $t$ extensions forms a $K_{t,t,t}$ in $G$, which is a contradiction.
\end{proof}

Using Lemma~\ref{lemma:fewtriangles}, we show that $K_{t,t,t}$-free graphs admit a vertex ordering in which the left-neighborhood of every vertex is sparse. This will play a key role in our proof of Theorem~\ref{theorem: coloring}.

\begin{lemma} \label{lemma:find degeneracy order}
    For every positive integer $t$, there is a constant $C\coloneqq C(t)$ such that the following holds. Let $G$ be a $K_{t,t,t}$-free graph with maximum degree at most $\Delta$. Then, there is an ordering of the vertices of $G$ such that every vertex $u\in V(G)$ is contained in at most $3C\Delta^{2-1/t^2}$ triangles $uvw$ with both $v$ and $w$ preceding $u$ in the ordering.
\end{lemma}

\begin{proof}
    Let $C$ be the constant given by Lemma \ref{lemma:fewtriangles}. Using induction on the number of vertices of $G$, we will prove Lemma~\ref{lemma:find degeneracy order} for this choice of $C$. The base case (where $G$ has at most one vertex) is trivial.

    Let $G$ be a $K_{t,t,t}$-free graph on $n$ vertices with maximum degree at most $\Delta$. There exists a vertex $z$ which is contained in fewer than $3C\Delta^{2-1/t^2}$ triangles, for otherwise $G$ would contain more than $Cn\Delta^{2-1/t^2}$ triangles, contradicting Lemma \ref{lemma:fewtriangles}.
    
    Let $G'\coloneqq G-z$. Now $G'$ is a $K_{t,t,t}$-free graph on $n-1$ vertices and has maximum degree at most~$\Delta$. Hence, by the induction hypothesis, $G'$ has a vertex-ordering $\sigma$ such that every vertex $u\in V(G')$ is in at most $3C\Delta^{2-1/t^2}$ triangles $uvw$ in $G'$ with both $v$ and $w$ preceding $u$ in $\sigma$. Now let $\tau$ be the ordering of $V(G)$ that lists the vertices of $G'$ first, in the order given by $\sigma$, followed by $z$.
 It is then straightforward to verify that every vertex $u \in V(G)$ lies in at most $3C\Delta^{2 - 1/t^2}$ triangles $uvw$ in $G$ in which both $v$ and $w$ precede $u$ in the ordering $\tau$. Indeed, for $u \in V(G')$, any such triangle must lie entirely within $G'$, and the statement follows by the choice of $\sigma$. For $u=z$, the statement follows immediately by the choice of $z$. This completes the induction step and the proof of the lemma.
\end{proof}
\section{Coloring $K_{t, t, t}$-free graphs}\label{section: coloring}

In this section, we will prove Theorem~\ref{theorem: coloring}, which we restate below for convenience.

\begin{theorem*}[{Restatement of Theorem~\ref{theorem: coloring}}]
    For all $t \in \N$, there exists $\Delta_0 \in \N$ such that the following holds for all $\Delta\geq \Delta_0$.
    Let $G$ be a $K_{t, t, t}$-free graph of maximum degree $\Delta$.
    Then
    \[\chi(G) = O\left(\frac{t^6\Delta}{\log \Delta}\right).\]
\end{theorem*}

As mentioned in \S\ref{subsection: proof overview}, our proof builds on the ideas of Alon, Krivelevich, and Sudakov \cite[Theorem 2.1]{AKSConjecture} regarding the coloring of locally sparse graphs. However, their result does not apply directly in our setting, as the neighborhood of a vertex in a $K_{t,t,t}$-free graph can be very dense.

We will use the following celebrated result of Johansson (see \cite[Ch. 14]{MolloyReed} for a textbook treatment of the argument; see also \cite{Molloy, bernshteyn2019johansson} for improvements in terms of the leading constant factor).

\begin{theorem}[{\cite{Joh_triangle}}]\label{theorem: johansson}
    Let $G$ be a $K_3$-free graph of maximum degree $\Delta$.
    Then,
    \[\chi(G) = O\left(\frac{\Delta}{\log \Delta}\right).\]
\end{theorem}

As we show at the end of this section, Theorem \ref{theorem: coloring} follows from the following key lemma, together with Theorem \ref{theorem: johansson}. Here and below, we omit floor and ceiling signs whenever this does not make substantial difference.

\begin{lemma}\label{lemma: partition triangle free}
    For each $t \in \N$, there exists some $\Delta_0\coloneqq \Delta_0(t)$ such that the following holds for every $\Delta\geq \Delta_0$. Let $G$ be a $K_{t, t, t}$-free graph of maximum degree $\Delta$.
    Then there exists a partition $V_1\sqcup \cdots \sqcup V_k$ of the vertex set of $G$ such that $k = 120t^4\Delta^{1 - 1/(10t^2)}$, and for every $1 \leq i \leq k$, the induced subgraph $G[V_i]$ is $K_3$-free and has maximum degree at most $2\Delta^{1/(10t^2)}$.
\end{lemma}

\begin{proof}\stepcounter{ForClaims}\renewcommand{\theForClaims}{\ref*{lemma: partition triangle free}}
    We will assume wherever necessary that $\Delta$ is sufficiently large in terms of $t$.

    By Lemma~\ref{lemma:find degeneracy order}, there exists an ordering $(v_1, \ldots, v_n)$ of $V(G)$ such that for each vertex $v_i$, the subgraph $G[N_L(v_i)]$ contains at most $3C\Delta^{2-1/t^2}$ edges, where $N_L(v_i) \coloneqq N(v_i) \cap \set{v_1, \ldots, v_{i-1}}$ is the ``left neighborhood'' of $v_i$.
    For the remainder of this proof, fix such an ordering.\footnote{We remark that the consideration of such an ordering and the fact that we work with left neighborhoods is the main difference between our argument and that of \cite[Lemma 2.3]{AKSConjecture}.}

    For each vertex $v$, we say that $u \in N_L(v)$ is \textit{bad} if $|N(u) \cap N_L(v)| \geq \Delta^{1 - 1/(2t^2)}$ and \textit{good} otherwise.
    A simple double counting argument shows that $v$ has at most $\Delta^{1-1/(5t^2)}$ bad neighbors in $N_L(v)$.
    
    For $\ell \coloneqq \Delta^{1-1/(10t^2)}$, construct a partition $U_1\sqcup \cdots \sqcup U_\ell$ by coloring each vertex independently and uniformly at random from $[\ell]$.
    Define the following bad events for each vertex:
    \begin{enumerate}
        \item $A_v$: more than $2\Delta/\ell$ vertices in $N(v)$ receive the same color as $v$. 
        \item $B_v$: more than $\kappa = 35t^2$ bad neighbors of $v$ receive the same color as $v$.
        \item $C_v$: the set of good neighbors of $v$ that receive the same color as $v$ span at least $\mu = 100t^4$ edges.
    \end{enumerate}
    We will use the \hyperref[theorem: LLL]{Lov\'asz Local Lemma} (Theorem~\ref{theorem: LLL}) to show that there is an outcome where none of the above bad events occur.
    Let us begin by computing the probabilities of each of these events.
    We will do so through a sequence of claims.

    \begin{claim}\label{claim: prob A_v}
        $\Pr[A_v] \leq \Delta^{-3}$.
    \end{claim}

    \begin{claimproof}
        Let $X_v$ be the random variable denoting the degree of $v$ in its color class.
        Clearly, $X_v$ is binomially distributed with mean $|N(v)|/\ell \leq \Delta/\ell$.
        By the \hyperref[theorem: chernoff]{Chernoff bound} (Theorem~\ref{theorem: chernoff}), we have
        \begin{align*}
            \Pr[A_v] &= \Pr[X_v \geq 2\Delta/\ell] \\
            &\leq \Pr[\mathrm{Bin}(\Delta,1/\ell) \geq 2\Delta/\ell] \\
            &\leq \Pr[|\mathrm{Bin}(\Delta,1/\ell) - \Delta/\ell| \geq \Delta/\ell] \\
            &\leq 2\exp\left(-\frac{\Delta}{3\ell}\right) \leq \Delta^{-3},
        \end{align*}
        as desired.
    \end{claimproof}
    
    \vspace{0.1in}

    \begin{claim}\label{claim: prob B_v}
        $\Pr[B_v] \leq \Delta^{-3}$.
    \end{claim}

    \begin{claimproof}
        As observed earlier, $v$ has at most $\Delta^{1-1/(5t^2)}$ bad neighbors.
        It follows that
        \begin{align*}
            \Pr[B_v] &\leq \binom{\Delta^{1-1/(5t^2)}}{\kappa}\ell^{-\kappa} \\
            &\leq \left(\frac{e\Delta^{1-1/(5t^2)}}{\kappa\,\ell}\right)^{\kappa} \leq \Delta^{-3},
        \end{align*}
        as desired.
    \end{claimproof}
    
    \vspace{0.1in}

    \begin{claim}\label{claim: prob C_v}
        $\Pr[C_v] \leq \Delta^{-3}$.
    \end{claim}

    \begin{claimproof}
        To bound $\Pr[C_v]$, we use the following simple observation made by Alon, Krivelevich, and Sudakov \cite{AKSConjecture}: 
        if a graph has at least $\mu$ edges, then 
        \begin{itemize}
            \item either there is a vertex of degree at least $\sqrt{\mu}$, or
            \item the graph is $\sqrt{\mu}$-edge-colorable by Vizing's theorem \cite{Vizing}, implying there is a matching of size $\sqrt{\mu}$.
        \end{itemize}
        Therefore, $C_v$ can occur only if
        \begin{enumerate}[label=(O\arabic*)]
            \item\label{item: matching} we have a matching of size $\sqrt{\mu}$ on the good neighbors of $v$ which have the same color as $v$, or
            \item\label{item: degree} there is a good neighbor $u$ of $v$ such that $u$ has at least $\sqrt{\mu}$ neighbors in $N_L(v)$ of the same color as $u$ and $v$.
        \end{enumerate}
         The probability of \ref{item: matching} is at most
         \[\binom{3C\Delta^{2-1/t^2}}{\sqrt{\mu}}\ell^{-2\sqrt{\mu}} \leq \left(\frac{3eC\Delta^{2-1/t^2}}{\ell^2\sqrt{\mu}}\right)^{\sqrt{\mu}} \leq \frac{\Delta^{-3}}{2},\]
         and the probability of \ref{item: degree} is at most
         \[\Delta\binom{\Delta^{1-1/(2t^2)}}{\sqrt{\mu}}\ell^{-\sqrt{\mu}} \leq \Delta\left(\frac{e\,\Delta^{1-1/(2t^2)}}{\ell\sqrt{\mu}}\right)^{\sqrt{\mu}} \leq \frac{\Delta^{-3}}{2}.\]
         Putting both cases together completes the proof.
    \end{claimproof}

    \vspace{0.1in}

    Note that each of the events $A_v$, $B_v$, and $C_v$ is independent of all but at most $3\Delta^2+2$ others, as it is independent of all events $A_u$, $B_u$, $C_u$ corresponding to vertices $u$ whose distance from $v$ is bigger than $2$.
    By Claims~\ref{claim: prob A_v}, \ref{claim: prob B_v}, and \ref{claim: prob C_v}, we may apply the \hyperref[theorem: LLL]{Lov\'asz Local Lemma}  (Theorem~\ref{theorem: LLL}) with $p = \Delta^{-3}$ and $d_{LLL} = 3\Delta^2+2$ to show that there exists an outcome such that none of the events $A_v$, $B_v$, $C_v$ occur.
    In particular, for each $v\in U_i$ we have $|N(v)\cap U_i|\leq 2\Delta/\ell=2\Delta^{1/(10t^2)}$ (since $A_v$ does not occur), and we can remove a set $S_{v, i}$ of at most $\kappa + \mu$ vertices from $G[N_L(v)\cap U_i]$ such that the remaining graph has no edges (since $B_v$ and $C_v$ do not occur).
    Fix such an outcome.

    Consider an arbitrary part $U_i$.
    Define a graph $H_i$ where $V(H_i) = U_i$ and $uv \in E(H_i)$ if and only if $u \in S_{v, i}$ or $v \in S_{u, i}$.
    Note that $H_i$ is $(\kappa + \mu)$-degenerate. Indeed, the degeneracy order is determined by the sub-ordering of $v_1,\dots,v_n$ on the set $U_i$, noting that $u\in S_{v,i}$ implies that $u$ precedes $v$ in that ordering. Consequently, $H_i$ is $(\kappa + \mu + 1)$-colorable, which implies that $V(H_i) = U_i$ can be partitioned into $\kappa + \mu + 1$ parts $U_i^j$, for $j \in [\kappa + \mu + 1]$, such that each part induces a $K_3$-free subgraph of $G$. Moreover, since $U_i^j\subseteq U_i$, the induced subgraph $G[U_i^j]$ has maximum degree at most $2\Delta^{1/(10t^2)}$. Since $(\kappa + \mu + 1) \ell \le 120 t^4 \Delta^{1 - 1/(10 t^2)} = k$, it follows that the sets $U_i^j$, for $j \in [\kappa + \mu + 1]$ and $i \in [\ell]$, form the desired partition of the vertex set of $G$. 
\end{proof}

We are now ready to prove Theorem~\ref{theorem: coloring}.

\begin{proof}[Proof of Theorem~\ref{theorem: coloring}]
    Fix the partition $V_1 \sqcup \cdots \sqcup V_k$ of the vertex set of $G$ given by Lemma~\ref{lemma: partition triangle free}.
    By Theorem~\ref{theorem: johansson}, we can color each subgraph $G[V_i]$ with at most
    \[O\left(\frac{2\Delta^{1/(10t^2)}}{\log\left(2\Delta^{1/(10t^2)}\right)}\right) = O\left(\frac{t^2\Delta^{1/(10t^2)}}{\log \Delta}\right)\]
    colors.
    Using a disjoint set of colors for each set $V_i$, we have
    \begin{align*}
        \chi(G) &= O\left(\frac{k\,t^2\Delta^{1/(10t^2)}}{\log \Delta}\right) \\
        &=O\left(\frac{t^6\Delta}{\log \Delta}\right),
    \end{align*}
    completing the proof.
\end{proof}

\section{Independent sets in $K_{t, t, t}$-free graphs}\label{section: ISET proof}

In this section, we will prove Theorem~\ref{theorem: iset}.
We derive it from the following statement:

\begin{theorem}\label{theorem: iset with large n}
    For all $\epsilon > 0$ and $t \in \N$, there exists $d_0 \in \N$ such that the following holds for all $d\geq d_0$.
    There exists $n_0 \in \N$ such that for all $n \geq n_0$ any $n$-vertex $K_{t, t, t}$-free graph $G$ of average degree $d$ satisfies
    \[\alpha(G) \geq \left(1 - \epsilon\right)\dfrac{n\log d}{d}.\]
\end{theorem}

We note that the assumption $n \ge n_0$ is required for our probabilistic estimates. To deduce Theorem~\ref{theorem: iset} from Theorem~\ref{theorem: iset with large n}, consider the case $n<n_0$ and set $r \coloneqq \left\lceil \frac{n_0}{n} \right\rceil$. Let $G'$ be the disjoint union of $r$ vertex-disjoint copies of $G$. Then $|V(G')| = rn \ge n_0$, so Theorem~\ref{theorem: iset with large n} applies to $G'$ and yields an independent set of the asserted size. Since $G'$ is a vertex-disjoint union, by the pigeonhole principle, some copy of $G$ in $G'$ contains an independent set of the desired size, completing the proof of Theorem~\ref{theorem: iset}.

To prove Theorem~\ref{theorem: iset with large n}, we use the following proposition. It produces an independent set in a sufficiently regular graph and guarantees that, after deleting this set together with its neighbors, the remaining graph has well-controlled average degree. This proposition relies crucially on Lemma~\ref{lemma:fewtriangles}, which bounds the number of triangles in $K_{t,t,t}$-free graphs with a given maximum degree $\Delta$.

\begin{proposition}\label{prop: alg iset}
    Fix $0 < \kappa < 1/3$ and $t\in\N$. There exists $d_0\in\N$ such that for every $d\ge d_0$ there exists $n_0\in\N$ such that for all $n\ge n_0$, the following holds.
    Let $H$ be any $K_{t,t,t}$-free graph satisfying
    \begin{enumerate}[label=\textup{(C\arabic*)},ref=\textup{C\arabic*}]
        \item $|V(H)| = n$,
        \item $d(H) = d$, and
        \item $\Delta(H) = \Delta \leq (1+\kappa)d$.
    \end{enumerate}
    Set $p\coloneqq \kappa/d$ and $\gamma\coloneqq (1-p)^{\Delta}$. Then there exist an independent set $I\subseteq V(H)$ and a subgraph $\tilde H\subseteq H-N[I]$ such that, for $\beta\coloneqq d^{-1/(20t^2)}$,
    \begin{enumerate}[label=\textup{(R\arabic*)},ref=\textup{R\arabic*}]
        \item\label{item: ISET lb} $|I|\ge (1-2\kappa)np$,
        \item\label{item: n' lb} $|V(\tilde H)|=(1\pm \beta)\,\gamma(1-p)\,n$, and
        \item\label{item: d' lb} $d(\tilde H)\le (1+4\beta)\,\gamma\,d(H)$.
    \end{enumerate}
\end{proposition}

In the remainder of this section, we present our nibble procedure and prove that it produces an independent set of the desired size (assuming Proposition~\ref{prop: alg iset}). We begin with an informal overview before giving the formal description. Initialize $H_1 \coloneqq G$. At the start of each iteration~$i$, the current graph~$H_i$ has $N_i$ vertices and average degree~$D_i$. We then consider two possible cases: \smallskip

\begin{enumerate}[label=(Case \arabic*),wide]
    \item\label{item: case high deg vtx}
    \textbf{Cleaning step:} If $H_i$ contains a vertex of degree greater than $(1 + \eps / 10)D_i$, remove such a vertex~$v$ and set $H_{i+1} \coloneqq H_i - v$, updating $N_{i+1}$ and $D_{i+1}$ accordingly while tracking how this removal affects the eventual size of the independent set. \smallskip

    \item\label{item: case small delta}
    \textbf{Nibble step:} If $\Delta(H_i) \le (1 + \eps / 10)D_i$, apply Proposition~\ref{prop: alg iset} to obtain an independent set $I_i \subseteq V(H_i)$ and a subgraph $H_{i+1} \subseteq H_i - N_{H_i}[I_i]$.
\end{enumerate} \smallskip

The nibble procedure terminates during the $T$-th iteration if
\begin{enumerate}[label=(T\arabic*)]
    \item\label{stop: D} $D_{T}$ is too ``small'',
    \item\label{stop: R} we have been in \ref{item: case high deg vtx} ``many'' times, or
    \item\label{stop: ISET} we have been in \ref{item: case small delta} ``many'' times.
\end{enumerate}
The choice of the above stopping criteria will become evident shortly after the statement of Proposition~\ref{prop: two stage nibble}.
The nibble procedure is formally described in Algorithm \ref{algorithm: nibble}, where the notions of ``many'' and ``small'' introduced above are made precise.
(The reader may recall \eqref{eq: termination avg deg} in the proof overview \S\ref{subsection: proof overview} for an informal description of these notions.)

\vspace{0.2cm}
\begin{breakablealgorithm}
\caption{The nibble procedure}\label{algorithm: nibble}
\begin{flushleft}
\textbf{Input}: An $n$-vertex graph $G$ having average degree $d$ and a parameter $\eps \in (0, 1)$. \\
\textbf{Output}: An independent set $I \subseteq V(G)$ and a subgraph $H \subseteq G - N[I]$.
\end{flushleft}

\begin{enumerate}[itemsep = .2cm]
    \item Initialize $H_1 = G$, $I = \emptyset$, $R_1 = 1$, $N_1 = n$, $D_1 = d$, $\tau = 0$, and $i = 1$.

    \item\label{step: terminate} If 
    \ref{stop: D} $D_i < d^{\eps/8}$,
    \ref{stop: R} $R_i > d^{\eps/20}$,
    or
    \ref{stop: ISET} $\tau = \left \lceil \dfrac{10(1-\eps/3)\log d}{(1+\eps/5)\eps} \right \rceil$,  
    return $(I, H_i)$.
    
    \item If not, do the following:
    
    \begin{enumerate}[itemsep = .2cm]
        \item\label{step: remove high degree} If $H_i$ has a vertex $v$ of degree greater than $(1+\eps/10)D_i$, then let $H_{i+1} \coloneqq H_i-v$, $N_{i+1} = N_i - 1$, $D_{i+1} = d(H_{i+1})$, and $R_{i+1} = R_i\cdot \frac{N_i}{N_i-1}$.
        
        \item\label{step: run alg iset} Else (that is, if $H_i$ has maximum degree at most $(1+\eps/10)D_i$), then apply Proposition~\ref{prop: alg iset} to $H_i$ with 
        \[\kappa = \frac{\eps}{10}, \qquad d = D_i, \qquad n = N_i.\]
        Let $I_i, \tilde H_i$ be the resulting independent set and subgraph respectively. Update $I$ to $I \cup I_i$ and $\tau$ to $\tau + 1$, and let $H_{i+1} = \tilde H_i$, $N_{i+1} = |V(\tilde H_i)|$, $D_{i+1} = d(\tilde H_i)$, and $R_{i+1} = R_i$.
        
    \end{enumerate}

    \item Update $i$ to $i + 1$ and return to step~\ref{step: terminate}.
\end{enumerate}

\end{breakablealgorithm}
\vspace{0.3cm}

In the above procedure, during each iteration, $I$ denotes the independent set constructed so far, $R_i$ measures the improvement in the ratio $N_i / D_i$ resulting from the cleaning steps described in~\ref{item: case high deg vtx} (see the proof overview in~\S\ref{subsection: proof overview} for a discussion of why this ratio is tracked), and $\tau$ records the number of independent sets extracted from $G$—that is, the number of nibble steps described in~\ref{item: case small delta}. In Lemma~\ref{lemma: lb on n/d and n}, we will show that $N_i \ge n / d^{1 - \eps / 4}$ for every iteration~$i$ of Algorithm~\ref{algorithm: nibble}. This bound is crucial to guarantee that $N_i$ remains sufficiently large to apply Proposition~\ref{prop: alg iset} in Step~\ref{step: run alg iset} of Algorithm~\ref{algorithm: nibble} during each iteration, thereby ensuring that the algorithm is well-defined. Furthermore, the algorithm clearly terminates since the number of vertices strictly decreases after every iteration.

The main result concerning our nibble procedure (Algorithm~\ref{algorithm: nibble}) is stated below and directly implies Theorem~\ref{theorem: iset with large n}. It shows that either the independent set 
$I$ constructed iteratively by the nibble procedure is sufficiently large, or the graph remaining at the end of the procedure contains an independent set of the desired size.

\begin{proposition}\label{prop: two stage nibble}
    Fix $\varepsilon \in (0,1)$ and $t \in \mathbb{N}$. Suppose $d$ is sufficiently large depending on $\varepsilon$ and $t$, and $n$ is sufficiently large depending on $d$. Let $G$ be an $n$-vertex $K_{t,t,t}$-free graph of average degree $d$. Let $(I, H)$ be the output obtained after running Algorithm~\ref{algorithm: nibble} with input $(G, \eps)$.
    Either $|I| \geq (1-\eps)\dfrac{n\log d}{d}$ or $\alpha(H) \geq (1-\eps)\dfrac{n\log d}{d}$.
\end{proposition}

The rest of this section is devoted to the proof of Proposition~\ref{prop: two stage nibble}.
Consider running Algorithm~\ref{algorithm: nibble} with input $(G, \eps)$.
We will assume, without loss of generality, that $\eps$ is sufficiently small and $n,\,d$ are sufficiently large for all arguments in this section.
Suppose the algorithm runs for $T$ iterations, i.e., we reach step~\ref{step: terminate} during the $T$-th iteration.
Additionally, let $S \subseteq [T - 1]$ be the set of iterations where we perform step~\ref{step: run alg iset}.
Note that, by the stopping criterion in step~\ref{step: terminate} of Algorithm~\ref{algorithm: nibble}, we obtain the following inequality.
\begin{align}\label{eq: tau bound}
    \tau = |S| \leq \left\lceil\dfrac{10(1-\eps/3)\log d}{(1+\eps/5)\eps}\right\rceil.
\end{align}

The following lemma plays a key role in our proof. It establishes that throughout the nibble procedure, if $R_i$ grows, the ratio $\frac{N_i}{D_i}$ improves substantially over its initial value $\frac{n}{d}$, and that $N_i$ remains sufficiently large (recall \eqref{eq: lemma on n/d and n} from the proof overview \S\ref{subsection: proof overview}).

\begin{lemma}\label{lemma: lb on n/d and n}
For all $i \in [T]$, we have
\begin{enumerate}[label=\textup{(S\arabic*)},ref=\textup{S\arabic*}]
    \item\label{eq: bound on n/d} $\displaystyle 
    \frac{N_i}{D_i} \ge \left(1-\frac{\eps}{5}\right)R_i^{\eps/20}\frac{n}{d}.$

    \item\label{eq: bound on n} $\displaystyle 
    N_i \ge \frac{n}{d^{1-\eps/4}}.$
\end{enumerate}
\end{lemma}

Before we prove the above lemma, we deduce Proposition~\ref{prop: two stage nibble} from it.

\subsection{Proof of Proposition~\ref{prop: two stage nibble} assuming Lemma~\ref{lemma: lb on n/d and n}}
\label{sec:propassuminglem}
We consider each termination condition in step~\ref{step: terminate} of Algorithm~\ref{algorithm: nibble} separately:
\begin{itemize}[itemsep=5pt]
    \item First, suppose that $D_T < d^{\eps/8}$ (corresponding to condition~\ref{stop: D}). We have
    \[\alpha(H) \geq \frac{N_T}{1 + D_T} \geq \frac{N_T}{1 + d^{\eps/8}} \overset{\eqref{eq: bound on n}}{\geq} \frac{n}{d^{1 - \eps/4}(1 + d^{\eps/8})} \gg \frac{n\log d}{d},\]
    as desired.
    
    \item Next, suppose that $R_T > d^{\eps/20}$ (corresponding to condition~\ref{stop: R}).
    Since we may assume $D_T \geq d^{\eps/8} \gg 1$, we have
    \[\alpha(H) \geq \frac{N_T}{1 + D_T} \geq \frac{N_T}{2D_T} \overset{\eqref{eq: bound on n/d}}{\geq} \left(1-\frac{\eps}{5}\right)R_T^{\eps/20}\frac{n}{2d} \gg \frac{n\log d}{d},\]
    by the lower bound on $R_T$, as desired.

    \item Finally, suppose that $\tau = \left\lceil\frac{10(1-\eps/3)\log d}{(1+\eps/5)\eps}\right\rceil$ (corresponding to condition~\ref{stop: ISET}).
   In this case, applying Proposition~\ref{prop: alg iset}, item~\eqref{item: ISET lb}, with parameters $\kappa = \frac{\eps}{10}$, $d = D_i$, and $n = N_i$, together with the fact that $R_i \ge 1$ for all $i$, we obtain the following:
    \begin{align*}
        |I| = \sum_{i \in S}|I_i| \overset{\eqref{item: ISET lb}}{\geq} \sum_{i \in S}\left(1 - \frac{\eps}{5}\right)\frac{\eps\,N_i}{10\,D_i}
        &\overset{\eqref{eq: bound on n/d}}{\geq} \sum_{i \in S}\left(1 - \frac{\eps}{5}\right)^2R_i^{\eps/20}\frac{\eps\,n}{10d} \\
        &~\geq \left(1 - \frac{\eps}{5}\right)^2\frac{\eps\tau\,n}{10d} \\
        &~\geq \left(1 - \frac{2\eps}{5}\right)\frac{\eps\tau\, n}{10d} \\
        &~\geq \left(1 - \eps\right)\frac{n\log d}{d},
    \end{align*}
    as desired.
\end{itemize}

This shows that, for every termination condition of Algorithm~\ref{algorithm: nibble}, either 
$|I| \ge (1 - \eps)\dfrac{n\log d}{d}$ or $\alpha(H) \ge (1 - \eps)\dfrac{n\log d}{d}$. 
This completes the proof of Proposition~\ref{prop: two stage nibble}, assuming Lemma~\ref{lemma: lb on n/d and n}. 

In the next subsection, we prove Lemma~\ref{lemma: lb on n/d and n}.

\subsection{Tracking $N_i/D_i$ and $N_i$ during the nibble and proving Lemma~\ref{lemma: lb on n/d and n}}

To prove Lemma~\ref{lemma: lb on n/d and n}, we first establish two simple auxiliary results concerning the cleaning step that removes high-degree vertices in \ref{item: case high deg vtx}. These results will be useful when analyzing $R_i$.

\begin{lemma} \label{lem:one cleaning step}
    Let $0<\eps<1$ and let $H$ be a graph with $N\geq 3$ vertices and average degree $D$.
    Assume that $H$ has a vertex $v$ of degree greater than $(1+\eps/10)D$. Let $H'\coloneqq H-v$. Write $N'\coloneqq N-1$ for the number of vertices of $H'$, and let $D'$ be the average degree of $H'$. Then
    \[\frac{N'}{D'}\geq \left(1+\frac{\eps}{5N}\right){\frac{N}{D}}.\]
\end{lemma}

\begin{proof}
    Note that 
    \[e(H')=e(H)-d_{H}(v)< \frac{ND}{2}-\left(1+\frac{\eps}{10}\right)D.\] 
    Hence, 
    \[D'= \frac{2e(H')}{N'}<\frac{(N-2-\eps/5)D}{N'},\] 
    implying
    \[\frac{N'}{D'} > \frac{(N')^2}{(N-2-\eps/5)D}=\frac{(N-1)^2}{(N-2-\eps/5)D}.\] 
    The lemma follows by observing that 
    \[(N-1)^2>(N-2-\eps/5)\left(1+\frac{\eps}{5N}\right)N. \qedhere\]
\end{proof}

\begin{corollary} \label{cor:one cleaning step ratio}
    Let $0<\eps<1$ and let $H$ be a graph with $N\geq 3$ vertices and average degree $D$. Assume that $H$ has a vertex $v$ of degree greater than $(1+\eps/10)D$. Let $H'\coloneqq H-v$. Write $N'\coloneqq N-1$ for the number of vertices of $H'$, and let $D'$ be the average degree of $H'$. Then
    \[\frac{N'/D'}{N/D}\geq \left(\frac{N}{N'}\right)^{\eps/20}.\]
\end{corollary}

\begin{proof}
    By Lemma \ref{lem:one cleaning step}, we have 
    \[\frac{N'/D'}{N/D}\geq 1+\frac{\eps}{5N}\geq e^{\frac{\eps}{10N}}.\] 
    Hence, it suffices to verify that $e^{\frac{2}{N}}\geq \frac{N}{N'}$, which follows from $e^{\frac{2}{N}}\geq 1+\frac{2}{N}$, $N'=N-1$, and $N\geq 3$.
\end{proof}

We are now ready to prove Lemma~\ref{lemma: lb on n/d and n}.

\begin{proof}[Proof of Lemma~\ref{lemma: lb on n/d and n}]
    We begin with the proof of the lower bound on $N_i/D_i$ in \eqref{eq: bound on n/d}.
    The claim is trivial for $i = 1$. Now, fix an arbitrary $i \in [T-1]$.  
    If $i \in S$, apply Proposition~\ref{prop: alg iset}~items~\eqref{item: n' lb}~and~\eqref{item: d' lb} with $D_i, N_i, \varepsilon/10$ playing the roles of $d, n, \kappa$, respectively;  
    and if $i \notin S$, apply Corollary~\ref{cor:one cleaning step ratio}.  
    This yields the following for $\beta_i \coloneqq D_i^{-1/(20t^2)}$:
    \begin{align*}
        \frac{N_{i+1}}{D_{i+1}} \geq \left\{\begin{array}{cc}
           \dfrac{(1-\beta_i)(1 - \eps/(10D_i))}{(1+4\beta_i)}\dfrac{N_{i}}{D_{i}} & \text{if } i \in S; \vspace{0.25cm} \\
            \left(\dfrac{N_i}{N_i - 1}\right)^{\eps/20}\dfrac{N_{i}}{D_{i}}  & \text{if } i \notin S.
        \end{array}\right.
    \end{align*}
    Since $i < T$, it must be that $D_i \geq d^{\eps/8}$; otherwise, Algorithm~\ref{algorithm: nibble} would have terminated at iteration $i < T$ (specifically at step~\ref{step: terminate}), contradicting the assumption that the algorithm continues past this point. Therefore, recursively applying the above, we obtain
    \begin{align*}
        \frac{N_i}{D_i} &\geq R_i^{\eps/20}\frac{n}{d}\prod_{j \in S,\, j < i}(1 - 10\beta_j) \\
        &\geq R_i^{\eps/20}\frac{n}{d}\exp\left(-20\sum_{j \in S,\, j < i}\beta_j\right) \\
        &\geq R_i^{\eps/20}\frac{n}{d}\exp\left(-\frac{|S|}{d^{\eps/(200t^2)}}\right),
    \end{align*}
    where we use the fact that $D_j \geq d^{\eps/8}$ for each $j<T$.
    The bound in \eqref{eq: bound on n/d} now follows by the upper bound on $|S|$ in \eqref{eq: tau bound} and the fact that $d$ is sufficiently large in terms of $\eps$ and $t$.

    Next, we turn our attention to proving the lower bound on $N_i$ in \eqref{eq: bound on n}.
    We will begin by showing that
    \begin{align}\label{eq: lb on N_T}
        N_i \geq \frac{n}{R_i}\exp\left(-\frac{(1+\eps/5)\eps\tau}{10}\right) \geq \frac{n}{R_{i}d^{1-\eps/3}}.
    \end{align}
    The last inequality follows by the bound on $\tau$ in \eqref{eq: tau bound}.
    
    Following a similar strategy as in the proof of \eqref{eq: bound on n/d}, we have the following as a result of Proposition~\ref{prop: alg iset}~item~\eqref{item: n' lb}, the fact that $\Delta(H_i) \leq (1+\eps/10)D_i$ if we are in step~\ref{step: run alg iset} (i.e., $i \in S$) in Algorithm \ref{algorithm: nibble}, and the definition of step~\ref{step: remove high degree} in Algorithm \ref{algorithm: nibble}:
    \[\forall 1 \leq i \leq T-1, \qquad \frac{N_{i+1}}{N_i} \geq \left\{\begin{array}{cc}
        (1 - \beta_i)\left(1 - \dfrac{\eps}{10D_i}\right)^{(1+\eps/10)D_i + 1} & \text{if } i \in S; \vspace{0.5cm} \\
        \dfrac{N_i - 1}{N_i}  & \text{if } i \notin S.
    \end{array}\right.\]

    Recursively applying the above inequality and using the definition of $R_i$, we have
    \begin{align*}
        N_{i} &\geq \frac{n}{R_{i}}\prod_{j \in S, j < i}\left(\left(1 - \frac{\eps}{10D_j}\right)^{(1+\eps/10)D_j + 1}(1 - \beta_j)\right) \\
        &\geq \frac{n}{R_{i}}\exp\left(- \frac{(1+\eps/8)\eps|S|}{10} - 2\sum_{j\in S}\beta_j\right) \\
        &\geq \frac{n}{R_{i}}\exp\left(- \frac{(1+\eps/5)\eps|S|}{10}\right),
    \end{align*}
    where the last two inequalities follow since $D_j \geq d^{\eps/8}$ for $j \in S$ and $d$ is sufficiently large in terms of $\eps$ and $t$.
    This completes the proof of \eqref{eq: lb on N_T}.

    For any given $i \le T$, we claim that if $R_{i} \leq d^{\eps/12}$, then \eqref{eq: bound on n} holds, as desired. Indeed, by \eqref{eq: lb on N_T}, we have
    \[N_i \geq \frac{n}{R_id^{1- \eps/3}} \geq \frac{n}{d^{1-\eps / 4}}.\]
   By the stopping criteria in step~\ref{step: terminate} of Algorithm~\ref{algorithm: nibble}, $R_i \leq d^{\eps/20} \ll d^{\eps/12}$ for all $i < T$.
    Therefore, $N_{T-1} \ge \frac{n}{d^{1-\eps / 4}}$.
    Hence, we only need to consider the case that $i = T$.
    To this end, we consider two cases based on the $(T - 1)$-th iteration of Algorithm~\ref{algorithm: nibble}:
    \begin{itemize}
        \item If we reach step~\ref{step: remove high degree}, then we have
        \begin{align*}
            R_{T} = R_{T - 1}\,\frac{N_{T - 1}}{N_{T - 1} - 1} 
            \leq \frac{d^{\eps / 20}}{1 - 1/N_{T - 1}} 
            \leq \frac{d^{\eps / 20}}{1 - d^{1-\eps/4}/n} \ll d^{\eps/12},
        \end{align*}
        as desired.

        \item If we reach step~\ref{step: run alg iset}, then $R_{T} = R_{T - 1} \leq d^{\eps / 20} \ll d^{\eps/12}$.
    \end{itemize}
    This covers all cases and completes the proof of the lemma.
\end{proof}

\section{Proof of Proposition~\ref{prop: alg iset}}\label{section: proof of prop}

In this section, we will prove Proposition~\ref{prop: alg iset}.
For the reader's convenience, we restate the proposition here.

\begin{proposition*}[{Restatement of Proposition~\ref{prop: alg iset}}]
    Fix $0 < \kappa < 1/3$ and $t\in\N$. There exists $d_0\in\N$ such that for every $d\ge d_0$ there exists $n_0\in\N$ such that for all $n\ge n_0$, the following holds.
    Let $H$ be any $K_{t,t,t}$-free graph satisfying
    \begin{enumerate}[label=\textup{(C\arabic*)}]
        \item \label{C1} $|V(H)| = n$,
        \item \label{C2} $d(H) = d$, and
        \item \label{C3} $\Delta(H) = \Delta \leq (1+\kappa)d$.
    \end{enumerate}
    Set $p\coloneqq \kappa/d$ and $\gamma\coloneqq (1-p)^{\Delta}$. Then there exist an independent set $I\subseteq V(H)$ and a subgraph $\tilde H\subseteq H-N[I]$ such that, for $\beta\coloneqq d^{-1/(20t^2)}$,
    \begin{enumerate}[label=\textup{(R\arabic*)}]
        \item \label{cond: iset lb} $|I|\ge (1-2\kappa)np$,
        \item \label{cond: tilde h vtx} $|V(\tilde H)|=(1\pm \beta)\,\gamma(1-p)\,n$, and
        \item \label{cond: tilde H avg deg} $d(\tilde H)\le (1+4\beta)\,\gamma\,d(H)$.
    \end{enumerate}
\end{proposition*}
Recall that here $N[I]$ denotes the closed neighborhood of $I$ in $H$, and the shorthand $x=(1\pm\beta)y$ means $(1-\beta)y\le x\le(1+\beta)y$. 

Let $H$ be a $K_{t,t,t}$-free graph satisfying conditions~\ref{C1}–\ref{C3}, and let $p$ and $\gamma$ be parameters as specified in Proposition~\ref{prop: alg iset}. To prove Proposition~\ref{prop: alg iset}, we apply Algorithm~\ref{algorithm: isetstep} (stated below) to $H$. The output of the algorithm then provides the desired independent set $I$ and subgraph $\tilde{H}$.

\vspace{0.2cm}
\begin{breakablealgorithm}
\caption{Independent Set Step}\label{algorithm: isetstep}
\begin{flushleft}
\textbf{Input}: A graph $H$ and parameters $p, \gamma \in [0,1]$. \\
\textbf{Output}: An independent set $I \subseteq V(H)$ and a subgraph $\tilde{H} \subseteq H - N[I]$.
\end{flushleft}
\begin{enumerate}[itemsep = .2cm, label = (\arabic*)]
    \item Sample $A \subseteq V(H)$ as follows: for each $v \in V(H)$, include $v \in A$ independently with probability $p$.
    We call $A$ the set of \emphd{activated} vertices.

    \item\label{step: eq} Let $\{\eq(v) \,:\, v \in V(H)\}$ be a family of independent random variables with distribution
    \[\eq(v) \sim \mathsf{Ber}\left(\frac{\gamma}{\left(1 - p\right)^{\deg(v)}}\right).\]
    Call $\eq(v)$ the \emphd{equalizing coin flip} for $v$.
    
    \item Sample $K \subseteq V(H)$ as follows: for each $v \in V(H)$, include $v \in K$ if $\eq(v) = 1$ and $N[v] \cap A = \0$.

    \item Let $\tilde{H} \coloneqq H[K]$ and let $I$ be an independent set of maximum size in $H[A]$.

\end{enumerate}
\end{breakablealgorithm}
\vspace{0.3cm}

Note that this algorithm is well-defined since $\gamma = (1-p)^{\Delta(H)}$.

We will prove Proposition~\ref{prop: alg iset} using the following two lemmas. The first lemma shows that the independent set $I$ produced by Algorithm~\ref{algorithm: isetstep} is sufficiently large with high probability, thereby establishing that condition~\ref{cond: iset lb} holds with high probability.

\begin{lemma}\label{lemma: large I}
    $\Pr[|I| \leq (1-2\kappa)np] = \exp\left(-\Omega\left(\dfrac{\kappa^4n}{d^5}\right)\right)$.
\end{lemma}

We prove Lemma~\ref{lemma: large I} in \S\ref{subsection: analysis of I}. The second lemma shows that $\tilde H$ satisfies conditions \ref{cond: tilde h vtx} and \ref{cond: tilde H avg deg} with high probability. 

\begin{lemma}\label{lemma: tilde H conds}
    $\Pr[\text{one of the conditions \ref{cond: tilde h vtx} or \ref{cond: tilde H avg deg} is violated }] = \exp\left(-\Omega\left(\dfrac{n}{d^5}\right)\right)$.
\end{lemma}

We prove Lemma~\ref{lemma: tilde H conds} in \S\ref{subsection: tilde H}.
With  Lemmas~\ref{lemma: large I} and ~\ref{lemma: tilde H conds} in hand, we can easily prove Proposition~\ref{prop: alg iset} as follows.

\begin{proof}[Proof of Proposition~\ref{prop: alg iset}]
    We have
    \begin{align*}
        &~\Pr[\text{one of the conditions \ref{cond: iset lb}--\ref{cond: tilde H avg deg} is violated}] \\
        &\leq \Pr[|I| \leq (1-2\kappa)np] + \Pr[\text{one of the conditions \ref{cond: tilde h vtx} or \ref{cond: tilde H avg deg} is violated}] \\
        &= \exp\left(-\Omega\left(\frac{\kappa^4n}{d^5}\right)\right).
    \end{align*}
    In particular, there exists an outcome $(I, \tilde H)$ of Algorithm~\ref{algorithm: isetstep} such that all conditions \ref{cond: iset lb}--\ref{cond: tilde H avg deg} are satisfied, completing the proof of Proposition~\ref{prop: alg iset}.
\end{proof}

It remains to prove Lemmas~\ref{lemma: large I} and ~\ref{lemma: tilde H conds}.

\subsection{Proof of Lemma~\ref{lemma: large I}: establishing that~\ref{cond: iset lb} holds with high probability}\label{subsection: analysis of I}

In this subsection, we will prove Lemma~\ref{lemma: large I}.
First, we show that $|A|$ is large with high probability.

\begin{lemma}\label{lemma: |A| concentration}
    $\Pr[|A| \leq (1 - pd)np] = \exp\left(-\Omega\left(\dfrac{\kappa^3n}{d}\right)\right)$.
\end{lemma}

\begin{proof}
    As $A$ is a $p$-random subset of $V(H)$, we have $\E[|A|] = np$.
    By the \hyperref[theorem: chernoff]{Chernoff bound} (Theorem~\ref{theorem: chernoff}), we have
    \begin{align*}
        \Pr\left[|A| \leq (1 - pd)np\right] &\leq \Pr\left[\big||A| - \E[|A|]\big| \geq pd\E[|A|]\right] \\
        &\leq 2\exp\left(-\frac{p^2d^2\E[|A|]}{3}\right) \\
        &= 2\exp\left(-\frac{np^3d^2}{3}\right) \\
        &= \exp\left(-\Omega\left(\frac{\kappa^3n}{d}\right)\right),
    \end{align*}
    as desired.
\end{proof}

Next, we show that $A$ is close to being an independent set with high probability.

\begin{lemma}\label{lemma: e(A) concentration}
    $\Pr[e(H[A]) \geq (1+p)p^2\,e(H)] = \exp\left(-\Omega\left(\dfrac{\kappa^4n}{d^5}\right)\right)$.
\end{lemma}

\begin{proof}
    Note that an edge $uv \in E(H)$ is in $e(H[A])$ if and only if $u, v \in A$.
    It follows that $\E[e(H[A])] = p^2e(H)$. To prove the lemma, we apply \hyperref[theorem: Talagrand]{Talagrand’s inequality} (Theorem~\ref{theorem: Talagrand}), where the random trials correspond to whether each vertex $v \in V(H)$ is included in $A$. Changing the membership of a single vertex in $A$ can affect $e(H[A])$ by at most $\Delta \le 2d$ (by \ref{C3} and since $\kappa < 1/3$). In addition, whether an edge $uv$ lies in $e(H[A])$ depends only on the inclusion of $u$ and $v$ in $A$.
    Hence, we may apply \hyperref[theorem: Talagrand]{Talagrand's inequality} (Theorem~\ref{theorem: Talagrand}) with $\mu = 2d$, $t = p\E[e(H[A])]/2$ and $r = 2$ to obtain:
    \begin{align*}
        &~\Pr\left[e(H[A]) \geq (1+p)p^2\,e(H)\right] \\
        &\leq \Pr\left[|e(H[A]) - \E[e(H[A])]| \geq p\E[e(H[A])]\right] \\
        &\leq \Pr\left[|e(H[A]) - \E[e(H[A])]| \geq p\E[e(H[A])]/2 + 40d\sqrt{2p^2e(H)} + 512 d^2\right] \\
        &= \exp\left(-\Omega\left(\frac{p^2\E[e(H[A])]}{d^2}\right)\right) \\
        &= \exp\left(-\Omega\left(\frac{p^4e(H)}{d^2}\right)\right) \\
        &= \exp\left(-\Omega\left(\frac{\kappa^4n}{d^{5}}\right)\right),
    \end{align*}
    where in the last equality we used that $e(H) = nd/2$.
\end{proof}

By Lemmas~\ref{lemma: |A| concentration} and \ref{lemma: e(A) concentration}, with probability at least 
\[1 - \left(\exp\left(-\Omega\left(\dfrac{\kappa^3n}{d}\right)\right) + \exp\left(-\Omega\left(\dfrac{\kappa^4n}{d^5}\right)\right)\right) = 1 - \exp\left(-\Omega\left(\dfrac{\kappa^4n}{d^5}\right)\right),\]
we have
\begin{align*}
    |I| &\geq |A| - e(H[A]) \\ 
    &\geq(1 - pd)np - (1+p)p^2\,e(H) \\
    &= (1 - pd)np - \frac{(1+p)p^2nd}{2} \\
    &= np\left(1 - pd - \frac{(1+p)pd}{2}\right) \\
    &\geq (1-2pd)np,
\end{align*}
completing the proof of Lemma~\ref{lemma: large I}.

\subsection{Proof of Lemma~\ref{lemma: tilde H conds}}\label{subsection: tilde H}

Throughout this section we let $\beta = d^{-1/(20t^2)}$ as in the statement of Proposition~\ref{prop: alg iset}.

\subsubsection{Showing that \ref{cond: tilde h vtx} holds with high probability}
First, we show that $|K| = |V(\tilde H)|$ is sharply concentrated around its expected value, thereby establishing that condition~\ref{cond: tilde h vtx} holds with high probability.

\begin{lemma}\label{lemma: concentrate tilde H v}
    $\Pr[|K| \notin (1 \pm \beta)\gamma(1-p)|V(H)|] = \exp\left(-\Omega\left(\dfrac{n}{d^4}\right)\right)$.
\end{lemma}

\begin{proof}
    By the definition of $K$, a vertex $u \in V(H)$ belongs to $K$ if and only if $N[u] \cap A = \emptyset$ and $\eq(u) = 1$.
    As the events $u \in A$ and $\eq(u) = 1$ are mutually independent over all $u \in V(H)$, we have the following for every $v \in V(H)$.

    \[\Pr[v\in K] = (1-p)^{\deg(v) + 1}\Pr[\eq(v) = 1] = \gamma\,(1-p).\]
    Therefore, $\E[|K|] = \gamma(1-p)|V(H)| = \gamma(1-p)n$.

    To establish concentration for $|K|$, we apply \hyperref[theorem: Talagrand]{Talagrand’s inequality} (Theorem~\ref{theorem: Talagrand}).
    The random trials are the events $u \in A$ and $\eq(u) = 1$ for $u \in V(H)$. Changing whether $u \in A$ can affect $|K|$ by at most $\Delta + 1 \leq 2d + 1$.
    On the other hand, changing the outcome of an equalizing coin flip of a single vertex can affect $|K|$ by at most $1$.
    Furthermore, $u \in K$ can be certified by the events $v\not \in A$ for $v \in N[u]$ and the outcome of $\eq(u)$.
    Hence, we may apply \hyperref[theorem: Talagrand]{Talagrand's inequality} (Theorem~\ref{theorem: Talagrand}) with $\mu = 2d + 1$ and $r = 2d + 2$ to obtain:
    \begin{align*}
        &~\Pr\left[|K| \notin (1 \pm \beta)\gamma(1-p)|V(H)|\right] \\
        &= \Pr\left[\big||K| - \E[|K|]\big| \geq \beta\E[|K|]\right] \\
        &\leq \Pr\left[\big||K| - \E[|K|]\big| \geq \beta\E[|K|]/2 + 20(2d + 1)\sqrt{(2d+2)\E[|K|]} + 2000d^3\right] \\
        &= \exp\left(-\Omega\left(\frac{\beta^2\E[|K|]}{d^3}\right)\right) \\
        &= \exp\left(-\Omega\left(\frac{\beta^2\gamma(1-p)n}{d^3}\right)\right).
    \end{align*}
  The lemma now follows since $\beta^2 \ge 1/d$, $1 - p \ge 1/2$, and $\gamma > 5/9$ whenever $\kappa < 1/3$ and $d$ is sufficiently large.
\end{proof}

\subsubsection{Showing that \ref{cond: tilde H avg deg} holds with high probability}

In the remainder of this subsection, we show that condition~\ref{cond: tilde H avg deg} also holds with high probability through a sequence of lemmas presented below.

The following lemma provides crucial control over $\mathbb{E}[e(\tilde H)]$ via Lemma~\ref{lemma:fewtriangles}, which states that every $K_{t,t,t}$-free graph on $n$ vertices with maximum degree at most $\Delta$ contains at most $C(t)n\Delta^{2-1/t^2}$ triangles.

\begin{lemma}\label{lemma: expectation of edges tilde H}
    $\E[e(\tilde H)] \leq \left(1+d^{-1/(15t^2)}\right)\gamma^2e(H)$.
\end{lemma}

\begin{proof}\stepcounter{ForClaims}\renewcommand{\theForClaims}{\ref*{lemma: expectation of edges tilde H}}
    First, we have
    \begin{align*}
        \E[e(\tilde H)] &= \sum_{uv \in e(H)}\Pr[u, v \in K] \\
        &= \sum_{uv \in e(H)}(1 - p)^{\deg(u) + \deg(v) - |N(u) \cap N(v)|}\Pr[\eq(u) = \eq(v) = 1] \\
        &= \gamma^2\sum_{uv \in e(H)}(1 - p)^{-|N(u) \cap N(v)|}.
    \end{align*}
  
    For the remainder of the proof, we define two subsets of edges of $H$ that partition $E(H)$.
    \[\Bad \coloneqq \left\{uv \in e(H)\,:\, |N(u) \cap N(v)| \geq d^{1 - 1/(5t^2)}\right\}, \qquad \Good \coloneqq e(H) \setminus \Bad.\]
    The following claim provides a bound on $|\Bad|$:

    \begin{claim}\label{claim: size of bad}
        $|\Bad| \leq nd^{1-3/(10t^2)}$
    \end{claim}

    \begin{claimproof}
        Note that 
        \[\sum_{uv \in e(H)}|N(u)\cap N(v)| = 3T(H),\]
        where $T(H)$ is the number of triangles in $H$.
        Since $H$ is $K_{t,t,t}$-free and satisfies $\Delta(H) \le 2d$, 
Lemma~\ref{lemma:fewtriangles} implies that 
\[
T(H) \le Cn(2d)^{2 - 1/t^2},
\]
for some constant $C = C(t)$.
 It follows that
        \[|\Bad|d^{1-1/(5t^2)} \leq 3Cn(2d)^{2-1/t^2} \implies |\Bad| \leq nd^{1 - 3/(10t^2)},\]
        for $d$ sufficiently large in terms of $t$. This proves the claim.
    \end{claimproof}

    With Claim~\ref{claim: size of bad} in hand, we have
    \begin{align*}
        \E[e(\tilde H)] &= \gamma^2\left(\sum_{uv\in \Good}(1 - p)^{-|N(u) \cap N(v)|} + \sum_{uv\in \Bad}(1 - p)^{-|N(u) \cap N(v)|}\right) \\
        &\leq \gamma^2\left(\sum_{uv\in \Good}(1 - p)^{-d^{1-1/(5t^2)}} + \sum_{uv\in \Bad}(1 - p)^{-\Delta}\right) \\
        &\leq \gamma^2\left(\frac{e(H)}{(1 - p)^{d^{1-1/(5t^2)}}} + \frac{nd^{1 - 3/(10t^2)}}{(1 - p)^{\Delta}}\right) \\
        &\leq \gamma^2e(H)\left(\frac{1}{(1 - p)^{d^{1-1/(5t^2)}}} + \frac{2d^{- 3/(10t^2)}}{(1 - p)^{\Delta}}\right) \\
        &\leq \gamma^2e(H) \left(1 + d^{- 1/(5t^2)} + 4d^{- 3/(10t^2)}\right) \\
        &\leq \left(1+d^{-1/(15t^2)}\right)\gamma^2e(H),
    \end{align*}
    where we used the fact that $\kappa < 1/3$, $nd = 2e(H)$, and $d$ is sufficiently large in terms of $t$.
\end{proof}

In the following lemma, we show that with high probability $e(\tilde H)$ is not too large.

\begin{lemma}\label{lemma: concentration of edges}
    $\Pr[e(\tilde H) \geq (1+\beta)\gamma^2e(H)] = \exp\left(-\Omega\left(\dfrac{n}{d^5}\right)\right)$.
\end{lemma}

\begin{proof}
    We prove the lemma using \hyperref[theorem: Talagrand]{Talagrand’s inequality} (Theorem~\ref{theorem: Talagrand}).
    The random trials correspond to the events ${u \in A}$ and ${\eq(u) = 1}$ for each $u \in V(H)$.
    Changing whether a vertex belongs to $A$ or altering the outcome of its equalizing coin flip can affect $e(\tilde H)$ by at most $\Delta^2 \le 4d^2$.
    Moreover, each event $uv \in e(\tilde H)$ can be certified by the events $x\not \in A$ for all $x \in N(u) \cup N(v)$ together with the outcomes of $\eq(u)$ and $\eq(v)$—a total of at most $2\Delta + 2 \le 4d + 2$ events.
    Hence, we may apply \hyperref[theorem: Talagrand]{Talagrand’s inequality} (Theorem~\ref{theorem: Talagrand}) with parameters $\mu = 4d^2$ and $r = 4d + 2$ to obtain:
    \begin{align*}
        &~\Pr[e(\tilde H) \geq (1+\beta)\gamma^2e(H)] \\
        &\leq \Pr\left[e(\tilde H) \geq (\beta - d^{-1/(15t^2)})\gamma^2e(H) + \E[e(\tilde H)]\right] \\
        &\leq \Pr\left[\big|e(\tilde H) - \E[e(\tilde H)]\big| \geq (\beta - d^{-1/(15t^2)})\gamma^2e(H)\right] \\
        &\leq \Pr\left[\big|e(\tilde H) - \E[e(\tilde H)]\big| \geq (\beta - d^{-1/(15t^2)})\gamma^2e(H)/2 + 80d^2\sqrt{(4d+2)\E[e(\tilde H)]} + 6000d^5\right] \\
        &= 2\exp\left(-\Omega\left(\frac{(\beta - d^{-1/(15t^2)})^2\gamma^2e(H)}{d^5}\right)\right),
    \end{align*}
    where we used Lemma~\ref{lemma: expectation of edges tilde H} and that $n$ is sufficiently large.
    The lemma now follows since $(\beta - d^{-1/(15t^2)})^2 \ge 1/d$, $e(H) = nd/2$ and $\gamma > 5/9$ whenever $\kappa < 1/3$.
\end{proof}

We now combine Lemmas~\ref{lemma: concentrate tilde H v} and~\ref{lemma: concentration of edges} to show that \ref{cond: tilde H avg deg} holds with high probability.

\begin{lemma}\label{lemma: tilde G and G avg deg}
    $\Pr[d(\tilde H) \geq (1 + 4\beta)\gamma d(H)] = \exp\left(-\Omega\left(\dfrac{n}{d^5}\right)\right)$.
\end{lemma}

\begin{proof}
    By Lemmas~\ref{lemma: concentrate tilde H v} and \ref{lemma: concentration of edges}, we have
    \[d(\tilde H) =  \frac{2e(\tilde H)}{|V(\tilde H)|} \leq \frac{2(1 + \beta)\gamma^2 e(H)}{(1 - \beta) \gamma(1-p)n} \leq (1 + 4\beta)\gamma d(H),\]
    with probability at least 
    \[1 - \exp\left(-\Omega\left(\dfrac{n}{d^4}\right)\right) - \exp\left(-\Omega\left(\dfrac{n}{d^5}\right)\right) \,=\, 1 - \exp\left(-\Omega\left(\dfrac{n}{d^5}\right)\right),\]
    completing the proof of the lemma.
\end{proof}

It is easy to see that Lemma~\ref{lemma: tilde H conds} now follows from Lemmas~\ref{lemma: concentrate tilde H v} and \ref{lemma: tilde G and G avg deg}, as desired.

\section{Concluding Remarks}\label{section: concluding remarks}

We conclude with a few potential avenues for future work.
In this paper, we introduced a new variant of the R\"odl nibble method in order to prove the lower bound on $\alpha(G)$ in Theorem~\ref{theorem: iset} (see Algorithm~\ref{algorithm: nibble} in \S\ref{section: ISET proof}).
The key technical novelty of this nibble procedure is that it alternates between two complementary steps: a \emph{cleaning} step and a \emph{nibble} step.
At its core, these steps work as follows:
\begin{itemize}
    \item if the input structure is ``good'', we perform the nibble step to extract a desired substructure, and
    \item if the input structure is ``bad'', we perform the cleaning step to transform the input into a ``better'' structure while carefully tracking its effect on the desired substructure.
\end{itemize}
In our application, the structure is a $K_{t, t, t}$-free graph, the substructure is an independent set, and the input graph is ``good'' if the maximum degree is not much larger than the average degree.
This procedure provides a template to tackle other problems, so it is of independent interest.

We note that our proof of Theorem~\ref{theorem: coloring} does not immediately extend to the list/correspondence coloring setting, and we leave it as an open problem.
Additionally, the upper bound on $c_2(F)$ in Theorem~\ref{theorem: general} is far from the lower bound of $1/2$ due to Bollob\'as \cite{BollobasIndependence}.
In \cite{anderson2025coloring}, Anderson, Bernshteyn, and 
the first named author of this manuscript 
asked whether the constant $c_2(F)$ is independent of $F$; they proved that $c_2(F) \leq 4$ for $F \subseteq K_{1, t, t}$ (see Table~\ref{table:bounds}).
In light of Theorem~\ref{theorem: iset}, we pose the following conjecture:

\begin{conjecture}\label{conj: Ktt AKS const 1}
    For any $3$-colorable graph $F$, every $F$-free graph $G$ with maximum degree $\Delta$ satisfies
\[
\chi(G) \le (1+o(1))\,\frac{\Delta}{\log \Delta}.
\]
\end{conjecture}

The above is known to hold for cycles, fans, and bipartite graphs $F$ (see Table~\ref{table:bounds}).
We suspect that the ideas developed in this paper could be used to resolve the conjecture for $F \subseteq K_{1,t,t}$; however, it is not clear how to adapt the approach directly.
Indeed, let us consider the general template discussed earlier. 
For the discussion that follows, we assume the reader is familiar with the coloring variant of the nibble method introduced by Kim in~\cite{Kim95}\footnote{See~\cite[Ch.~12]{MolloyReed} for a textbook exposition of the method, and~\cite{AndersonBernshteynDhawan, anderson2025coloring, alon2020palette, PS15} for more modern treatments.}.
Before describing the challenges in adapting our approach, we begin by introducing a few definitions. 
Let $G$ be a graph, and let $L$ be a list assignment for the vertices of $G$.
For each $v \in V(G)$ and $c \in L(v)$, the \emph{color-degree} of $c$ with respect to $v$ is defined as
\[
d_L(v, c) = \bigl|\{u \in N(v) : c \in L(u)\}\bigr|,
\]
and the \emph{average color-degree} at $v$ is
\[
d_L(v) = \frac{1}{|L(v)|} \sum_{c \in L(v)} d_L(v, c).
\]
To apply our template to the list-coloring problem, the structure is a $K_{1,t,t}$-free graph $G$ equipped with a list assignment $L$, the substructure is a partial $L$-coloring, and a pair $(G,L)$ is said to be \emph{good} if, for every vertex $v$, the maximum color-degree is not much larger than the average color-degree. In our setting, the cleaning step affects the global ratio $n/d$, whereas in the coloring variant the analogous quantity is $\ell/d$, where $\ell \coloneqq \min_v |L(v)|$ and $d \coloneqq \max_{v,c} d_L(v,c)$. Removing a vertex in Algorithm~\ref{algorithm: nibble} has a clear and predictable effect on the global parameter $n/d$. By contrast, deleting a ``high color-degree'' color from the list of a vertex $v$ changes the local ratio $|L(v)|/d_L(v)$ but may leave the global quantity $\ell/d$ unaffected.

In summary, the inherently local nature of the coloring variant of the nibble method prevents a direct adaptation of our approach, which depends on global control. Nevertheless, we believe the ideas developed here shed light on this problem and may represent progress toward a conjecture of Cambie and Kang~\cite{cambie2022independent}, which proposes a color-degree analogue of Molloy’s theorem on the list chromatic number of $K_3$-free graphs\footnote{The conjecture is stated in terms of correspondence colorings; see~\cite{bernshteyn2019johansson} for an extension of Molloy’s result to correspondence coloring.}~\cite{Molloy}. Additionally, we note that proving $c_1(K_{t, t, t}) < 1$ is very difficult even for $t = 1$ (see, e.g., \cite{Zdeborova, Achlioptas} and the discussion at the end of \S\ref{subsection: prior work}). 

\vspace{2mm}

\subsection*{Acknowledgments} We are very grateful to Anton Bernshteyn, Ross Kang, and Tom Kelly for their comments on a preliminary version of this paper.

\printbibliography

\end{document}